\documentclass[12pt]{amsart}
\usepackage{amsfonts}
\usepackage{amssymb}

\numberwithin{equation}{section}
\theoremstyle{plain}
 \newtheorem{thm}{Theorem}[section]
 \newtheorem{lem}[thm]{Lemma}
 \newtheorem{cor}[thm]{Corollary}
 \newtheorem{prop}[thm]{Proposition}
\theoremstyle{definition}
 \newtheorem{defn}[thm]{Definition}
 \newtheorem{ex}[thm]{Example}
 \newtheorem{rem}[thm]{Remark}

\newcommand{\gm}{\gamma}

\newcommand{\dl}{\delta}

\newcommand{\zt}{\zeta}

\newcommand{\ld}{\lambda}

\newcommand{\sg}{\sigma}

\newcommand{\ph}{\varphi}

\newcommand{\mcal}{\mathcal}

\newcommand{\wh}{\widehat}

\newcommand{\R}{\mathbb{R}}

\newcommand{\Z}{\mathbb{Z}}

\newcommand{\T}{\mathbb{T}}

\newcommand{\law}{\mathcal L}
\newcommand{\cl}{\colon}

\newcommand{\q}{\quad}
\newcommand{\qq}{\qquad}

\newcommand{\bN}{\mathbb{N}}
\newcommand{\bZ}{\mathbb{Z}}
\newcommand{\bR}{\mathbb{R}}
\newcommand{\bQ}{\mathbb{Q}}
\newcommand{\bC}{\mathbb{C}}
\newcommand{\cB}{\mathcal{B}}

\newcommand{\cA}{\mathcal{A}}

\newcommand{\supp}{\mbox{\rm supp }}
\newcommand{\qid}{\mbox{\rm q.i.d.}}
\newcommand{\bE}{\mathbb{E}}
\newcommand{\bL}{\mathbb{L}}

\newcommand{\re}{\mathrm{e}}
\newcommand{\ri}{\mathrm{i}}
\newcommand{\di}{\mathrm{d}}
\newcommand{\one}{\mathbf{1}}

\allowdisplaybreaks

\begin{document}
\author{Alexander Lindner \and Lei Pan \and Ken-iti Sato}

\title{On quasi-infinitely divisible distributions}
\maketitle

\begin{abstract}
A quasi-infinitely divisible distribution on $\mathbb{R}$ is a probability distribution whose characteristic function allows a L\'evy-Khintchine type representation with a  \lq\lq signed L\'evy measure\rq\rq, rather than a L\'evy measure.
Quasi-infinitely divisible distributions appear naturally in the factorization of infinitely divisible distributions. Namely, a distribution $\mu$ is quasi-infinitely divisible if and only if there are two infinitely divisible distributions $\mu_1$ and $\mu_2$ such that $\mu_1 \ast \mu = \mu_2$. The present paper studies certain properties of quasi-infinitely divisible distributions in terms of their characteristic triplet, such as properties of supports, finiteness of moments, continuity properties and weak convergence, with various examples constructed.
In particular, it is shown that the set of quasi-infinitely divisible distributions is dense in the set of all probability distributions with respect to weak convergence.  Further, it is proved that a distribution concentrated on the integers is quasi-infinitely divisible if and only if its characteristic function does not have zeroes, with the use of the Wiener-L\'evy theorem on absolutely convergent Fourier series.
A number of fine properties of such distributions are proved based on this fact.  A similar characterisation is not true for non-lattice  probability distributions on the line.
\end{abstract}

\section{Introduction}
The class of infinitely divisible distributions on the real line is well studied and completely characterized by the L\'evy-Khintchine formula. The aim of this paper is to obtain some results on quasi-infinitely divisible distributions, i.e. distributions whose characteristic functions allow a L\'evy-Khintchine type representation with  \lq\lq signed L\'evy measures\rq\rq~ rather than L\'evy measures. Such distributions have been considered and appeared before in various examples, in particular in connection with the problem of the factorization of distributions, by \cite{Cuppens1975,Linnik,LinnikOstrovskii} and others. Cuppens \cite{Cuppens1975} and Linnik and Ostrovski\u{i} \cite{LinnikOstrovskii} give extensive treatment of such distributions including the multidimensional case. The term \lq\lq quasi-infinitely divisible distribution\rq\rq~ for such distributions has been introduced in \cite{LiSa2011}. It should be noted that in the context of Poisson mixtures, Puri and Goldie \cite{PuriGoldie} also introduced the notion of quasi-infinitely divisible distributions, but this notion has nothing to do with the notion of quasi-infinitely divisible used in this paper.

To get the definitions right, recall that a distribution $\mu$ on $\bR$ is infinitely divisible if and only if for every $n\in \bN$ there exists a distribution $\mu_n$ on $\bR$ such that $\mu_n^{\ast n} = \mu$. The characteristic function of an infinitely divisible distribution $\mu$ can be expressed by the L\'evy-Khintchine formula. To state it, by a \emph{representation function} we mean a function $c\cl\bR \to \bR$ which is bounded, Borel measurable and satisfies
\begin{equation} \label{eq-trunc1}
\lim_{x\to 0}\,(c(x)-x)/x^2 =0.
\end{equation}
In this paper $c$ always denotes a representation function.
Then the L\'evy-Khintchine formula states that, when we fix a representation function $c$, a probability measure $\mu$ on $\bR$ is infinitely divisible if and only if its characteristic function $z \mapsto \widehat{\mu} (z) = \int_{\bR} \re^{\ri z x} \, \mu(\di x)$ can be expressed in the form
\begin{equation} \label{eq-trunc2}
\widehat{\mu} (z) = \exp \left( \Psi_\mu(z)\right) , \quad z\in \bR,
\end{equation}
where
\begin{equation} \label{eq-trunc3}
\Psi_\mu(z) = \ri \gamma z -\frac12 a z^2 + \int_\bR \left( \re^{\ri z x} - 1 - \ri z c(x) \right) \, \nu(\di x), \quad z\in \bR,
\end{equation}
with $a\geq 0$, $\gamma \in \bR$ and $\nu$ being a measure on $\R$ satisfying
\begin{equation}\label{eq-levy}
\nu(\{0\})=0\q\text{and}\q \int_{\R} (1\land x^2)\nu(\di x)<\infty.
\end{equation}
The triplet $(a,\nu,\gamma)$ is unique and called the \emph{characteristic triplet with respect to $c$}, while the function $\Psi_\mu$ is called the \emph{characteristic exponent} of $\mu$ and is the unique continuous function satisfying $\Psi_\mu(0) = 0$ and \eqref{eq-trunc2}.
The measure $\nu$ is called the \emph{L\'evy measure} of $\mu$ and the constant $a$ the \emph{Gaussian variance} of $\mu$; these two are
independent of the choice of $c$.  The constant $\gm$ depends on the choice of $c$ and thus $\gm$ is called \emph{$c$-location} of $\mu$.
More precisely, if $c_1$ and $c_2$ are two representation functions and $\gm_j$ is the $c_j$-location of $\mu$ for $j=1,2$,
then
\begin{equation} \label{eq-location}
\gamma_2 = \gamma_1 + \int_{\R} \big( c_2(x) - c_1(x) \big) \nu(\di x).
\end{equation}
Conversely, given $a\geq 0$, $\gamma \in \bR$, a measure $\nu$ on $\bR$ satisfying \eqref{eq-levy}
and a representation function $c$, the function $x\mapsto |\re^{\ri z x} - 1 - \ri z c(x)|$ is integrable with respect to $\nu$ for each $z\in \bR$ and the right-hand side of \eqref{eq-trunc2} together with \eqref{eq-trunc3} defines the characteristic function of an infinitely divisible distribution. The function $c$ is often chosen as
$c(x) = x \one_{[-1,1]} (x)$. All these facts are well known and can be found in Sections 7, 8 and 56 of Sato \cite{Sa}, for example. When working in one dimension as we do here, it is often more convenient to combine $\nu$ and $a$ into a single measure. More precisely, let $c$ be a representation function
and define the function
\begin{equation} \label{eq-def-g}
g_{c}\cl \bR \times \bR \to \bC \quad \mbox{by}  \quad g_c(x,z) = \begin{cases}
(\re^{\ri z x} - 1 - \ri z c(x))/(1\wedge x^2), & x \neq 0, \\
- z^2/2, & x=0 .
\end{cases}
\end{equation}
Observe that $g_c(\cdot , z)$ is  bounded for each fixed $z\in \R$, and it is continuous at $0$, which follows from \eqref{eq-trunc1}.
Now, if $\mu$ is infinitely divisible with characteristic triplet $(a,\nu,\gamma)$ with respect to $c$, then $\widehat{\mu}$ has the representation
\begin{equation} \label{eq-trunc5}
\widehat{\mu} (z) = \exp \left( \ri \gamma z + \int_\bR g_c(x,z) \, \zeta(\di x) \right), \quad z\in \bR,
\end{equation}
where the measure $\zeta$ on $\bR$ is finite and given by
\begin{equation} \label{eq-trunc6}
\zeta (\di x) = a \delta_0 (\di x)+ (1 \wedge x^2) \, \nu(\di x),
\end{equation}
with $\delta_0$ denoting the Dirac measure at $0$. Conversely, to any finite measure $\zeta$ on $\bR$ we can define
$a$ and $\nu$ by $a= \zeta (\{0\})$ and $\nu(\di x) = (1\wedge x^2)^{-1} \one_{\bR \setminus \{0\}} (x) \zeta(\di x)$. We shall hence speak of $(\zeta,\gamma)$ as the \emph{characteristic pair} of $\mu$ with respect to $c$. The characteristic pair is obviously unique for given $c$ and $\zt$ is independent of the choice of $c$. With these preparations, we can now define quasi-infinitely divisible distributions:

\begin{defn} \label{def1-qid}
Let $c$ be a fixed representation function.
A distribution $\mu$ on $\bR$ is \emph{quasi-infinitely divisible}, if its characteristic function admits the representation \eqref{eq-trunc5} with some $\gamma \in \bR$ and a finite \emph{signed} measure $\zeta$ on $\bR$. The pair $(\zeta,\gamma)$ is then called the \emph{characteristic pair} of $\mu$ with respect to $c$, and
$\Psi_\mu$, defined by $\Psi_\mu (z) = \ri \gamma z + \int_\bR g_c(x,z) \, \zeta(\di x)$, satisfies \eqref{eq-trunc2} and is called
the \emph{characteristic exponent} of $\mu$.
\end{defn}

Recall that a \emph{signed measure} $\zeta$ on $\bR$ is a function $\zeta\cl \cB \to [-\infty,\infty]$ on the Borel $\sigma$-algebra $\cB$ such that $\zeta(\emptyset) = 0$ and $\zeta(\bigcup_{j=1}^\infty A_j) = \sum_{j=1}^\infty \zeta(A_j)$ for all sequences $(A_j)_{j\in\bN}$ of pairwise disjoint sets in $\cB$, where the infinite series converges in $[-\infty,\infty]$; in particular, the value of the series  does not depend on the order of the $A_j$, i.e. the series converges unconditionally. A signed measure $\zeta$ is \emph{finite}, if $\zeta(A) \in \bR$ for all $A\in \cB$. Similarly to infinitely divisible distributions, the characteristic exponent $\Psi_\mu$ of $\mu$ is the unique continuous function satisfying $\Psi_\mu(0) = 0$ and  \eqref{eq-trunc2}, and the characteristic pair of a quasi-infinitely divisible distribution is unique for a fixed function $c$, see e.g.\ Linnik \cite[Thm.~6.1.1]{Linnik}, Cuppens \cite[Thm.~4.3.3]{Cuppens1975} or Sato \cite[Exercise 12.2]{Sa}; further, it is easy to see that if $c_1$ and $c_2$ are two representation functions and $(\zeta_1,\gamma_1)$ and $(\zeta_2,\gamma_2)$ are the characteristic pairs with respect to $c_1$ and $c_2$, respectively, then $\zeta_1 = \zeta_2$ and
$$
\gm_2=\gm_1 +\int_{\R\setminus\{0\}} (1\land x^2)^{-1} (c_2(x)-c_1(x)) \zt_1(\di x).
$$
It is clear that not every pair $(\zeta,\gamma)$ with $\zeta$ being a finite signed measure which is not  positive gives rise to a quasi-infinitely divisible distribution; for otherwise, with $(\zeta,\gamma)$ being the characteristic pair of a quasi-infinitely divisible distribution $\mu$, also $(n^{-1} \zeta, n^{-1} \gamma)$ would be the characteristic pair of a quasi-infinitely divisible distribution $\mu_n$ for each $n\in \bN$, and $\mu_n^{\ast n} = \mu$, showing that $\mu$ is infinitely divisible, hence $\zeta$ must be positive  by the uniqueness of the characteristic pair, which is absurd.
The question which pair $(\zeta,\gamma)$ gives rise to a distribution is a difficult one, and a very related question (when the associated quasi-L\'evy type measure is finite) was already posed  by Cuppens \cite[Section 5]{Cuppens1969}. We do not provide an answer to this question, but will give some examples of quasi-infinitely divisible distributions and also study properties of the distribution in terms of the characteristic pair.

Quasi-infinitely divisible distributions arise naturally in the study of factorization of probability distribution. To see that, observe that the difference of two finite measures is a finite signed measure. Recall that for a signed measure $\zeta$ on $\bR$, the \emph{total variation} of $\zeta$ is the measure $|\zeta|\cl \cB \to [0,\infty]$ defined by
\begin{equation} \label{eq-variation}
|\zeta|(A) = \sup \sum_{j=1}^\infty |\zeta(A_j)|,
\end{equation}
where the supremum is taken over all partitions $\{A_j \}$ of $A\in \cB$. The total variation $|\zeta|$ is finite if and only if $\zeta$ is finite.
Further, by the Hahn-Jordan decomposition, for a finite signed measure $\zeta$, there exist disjoint Borel sets $C^+$ and $C^-$ and finite   measures $\zeta^+$ and $\zeta^-$ on $\cB$ with $\zeta^+(\bR \setminus C^+) = \zeta^-(\bR \setminus C^-) = 0$ and $\zeta = \zeta^+ - \zeta^-$, and the measures $\zeta^+$ and $\zeta^-$ are uniquely determined by $\zeta$. It holds
\begin{equation} \label{eq-Hahn}
\zeta^+ = \frac12 (|\zeta| + \zeta), \quad \zeta^- = \frac12 ( |\zeta| - \zeta), \quad |\zeta| = \zeta^+ + \zeta^-.
\end{equation}
Now if $\mu$ is quasi-infinitely divisible with characteristic pair $(\zeta,\gamma)$ with respect to a function $c$, define the infinitely divisible distributions $\mu^+$ and $\mu^-$ to have characteristic pairs $(\zeta^+,\gamma)$ and $(\zeta^-, 0)$, respectively. Since $\zeta + \zeta^- = \zeta^+$, it follows that $\Psi_{\mu^+} (z) = \Psi_{\mu}(z) + \Psi_{\mu^-}(z)$, i.e. $\widehat{\mu^+}(z) = \widehat{\mu}(z) \widehat{\mu^-}(z)$. So if $\mu$ is quasi-infinitely divisible, there exist two infinitely divisible distributions $\mu_1$ and $\mu_2$ such that $\widehat{\mu_1}(z) = \widehat{\mu_2}(z) \widehat{\mu}(z)$, i.e. such that $\mu$ and $\mu_2$ factorize $\mu_1$. On the other hand, if a distribution $\mu$ is such that two infinitely divisible distributions $\mu_1$ and $\mu_2$ with characteristic pairs $(\zeta_1,\gamma_1)$ and $(\zeta_2,\gamma_2)$ exist with $\widehat{\mu_1}(z) = \widehat{\mu_2}(z) \widehat{\mu}(z)$, then $\widehat{\mu_2}(z) \neq 0$ for all $z\in \bR$ and
$$\widehat{\mu}(z) = \frac{\widehat{\mu_1}(z)}{\widehat{\mu_2}(z)} = \exp \left( \ri (\gamma_1 - \gamma_2)  z + \int_{\bR} g_c(x,z) \, (\zeta_1 - \zeta_2)(\di x) \right),\quad z\in \bR,$$
showing that $\mu$ is quasi-infinitely divisible with characteristic pair $(\zeta_1-\zeta_2, \gamma_1 - \gamma_2)$. Summing up, a distribution $\mu$ is quasi-infinitely divisible if and only if there exist two infinitely divisible distributions $\mu_1$ and $\mu_2$ such that $\mu_2$ and $\mu$ factorize $\mu_1$, i.e. such that $\widehat{\mu_1}(z) = \widehat{\mu}(z) \widehat{\mu_2}(z)$. In terms of random variables, $\mu$ is quasi-infinitely divisible if and only if there exist random variables $X,Y,Z$  such that
\begin{equation} \label{eq-factor1}
 X+Y \stackrel{d}{=} Z,\quad \mbox{$X$ and $Y$ independent},
\end{equation}
and such that $\law(X)=\mu$ and $\law(Y)$ and $\law(Z)$ are infinitely divisible. The random variables $Y$ and $Z$ can then be chosen to have characteristic pairs $(\zeta^-,0)$ and $(\zeta^+,\gamma)$, respectively, if $(\zeta,\gamma)$ is the characteristic pair of $\mu$. This factorization property explains the interest in quasi-infinitely divisible distributions.

Apart from the decomposition problem of probability measures, quasi-infinitely divisible distributions appear
in the study of several problems in probability theory.  Some of them are mentioned with references in
Lindner and Sato \cite{LiSa2011} and in the solution of Exercise 12.4 of \cite{Sa}.  In relation to stochastic
processes, the stationary distribution of a generalized Ornstein-Uhlenbeck process associated with a
bivariate L\'evy process with three parameters can be infinitely divisible, non-infinitely divisible
quasi-infinitely divisible, or non-quasi-infinitely divisible, which is thoroughly analysed in \cite{LiSa2011}.

The goal of this paper is to study properties of quasi-infinitely divisible distributions in terms of their characteristic pairs, or, equivalently, in terms of their characteristic triplets. The quasi-L\'evy measure and characteristic triplet will be introduced in the next section, along with some preliminary remarks about quasi-infinitely divisible distributions. Section \ref{S3} contains some examples of quasi-infinitely divisible distributions. In Section \ref{S4} we study convergence properties of a sequence of quasi-infinitely divisible distributions in terms of the characteristic pairs. Sections \ref{S5}, \ref{S6} and \ref{S7} are concerned with the supports, moments and continuity properties of quasi-infinitely divisible distributions, respectively. Finally, in Section \ref{S-integers} we specialise in  distributions concentrated on the integers, show that such a distribution is quasi-infinitely divisible if and only if its characteristic function has no zeroes, and derive sharper convergence and moment conditions for quasi-infinitely divisible distributions concentrated on the integers.

To fix notation (which partially has been already used), by a distribution on $\bR$ we mean a probability measure on $(\bR,\cB)$, with $\cB$ being the Borel $\sigma$-algebra on $\bR$, and similarly, by a signed measure on $\bR$ we mean it to be defined on $(\bR,\cB)$. By a measure on $\bR$ we always mean a positive measure on $(\bR,\cB)$, i.e. an $[0,\infty]$-valued $\sigma$-additive set-function on $\cB$ that assigns the value 0 to the empty set.  The Dirac measure at a point $b\in \bR$ will be denoted by $\delta_b$, the Gaussian distribution with mean $a\in \bR$ and variance $b\geq 0$ by $N(a,b)$. The support of a signed measure $\mu$ on $\bR$ is defined to be the support of its total variation $|\mu|$ and will be denoted by $\mbox{\rm supp} (\mu)$, the restriction of $\mu$ to a subset $\cA \subset \cB$ by $\mu_{|\cA}$, and for $A\in \cB$ we often write $\mu_{|A}$ for $\mu_{|A\cap \cB}$. Weak convergence of signed measures (as defined in Section \ref{S4}) will be denoted by \lq\lq $\stackrel{w}{\to}$\rq\rq, and the Fourier transform at $z\in \bR$ of a finite signed (or complex) measure $\mu$ on $\bR$ by $\widehat{\mu}(z) = \int_\bR \re^{\ri z x} \, \mu(\di x)$.  By $\bL_\mu (u) = \int_\bR \re^{-ux} \, \mu(\di x)\in [0,\infty]$  we denote the Laplace transform of a distribution $\mu$ on $\bR$ at $u\geq 0$, irrelevant if $\mu$ is concentrated on $[0,\infty)$ or not.
We say the Laplace transform is finite, if $\bL_\mu(u) < \infty$ for all $u\geq 0$, which is in particular the case when the support of $\mu$ is bounded from below.
The convolution of two finite signed  (or complex) measures $\mu_1$ and $\mu_2$ on $\bR$ is defined by $\mu_1 \ast \mu_2( B) = \int_\bR \mu_1( B-x) \, \mu_2(\di x)$, $B\in \cB$, where $B-x = \{ y-x: y\in B\}$, and the $n$-fold convolution of $\mu_1$ with itself is denoted by $\mu_1^{\ast n}$. See \cite[Sect. 2.5]{Cuppens1975} or Rudin \cite[Exercise 8.5]{Rudin} for more information on the convolution of finite signed or complex measures.
The law of a random variable $X$ will be denoted by $\law(X)$, and equality in distribution will be written as $X\stackrel{d}{=} Y$. The expectation of a random variable $X$ is denoted by $\bE X$, its variance by $\mbox{\rm Var} (X)$. We write $x\wedge y = \min \{x, y\}$ and $x\vee y = \max \{x,y\}$ for $x,y\in \bR$. The real and imaginary part of a complex number $w$ will be denoted by $\Re (w)$ and $\Im (w)$, respectively, and by $\ri$ we denote the imaginary unit.
We write $\bN = \{1,2,\ldots\}$, $\bN_0 = \bN \cup \{0\}$ and $\bZ$, $\bQ$, $\bR$ and $\bC$ for the set of integers, rational numbers, real numbers and complex numbers, respectively.
 The indicator function of a set $A\subset \bR$ is denoted by $\one_A$.


\section{Quasi-L\'evy measures and first remarks} \label{S2}

Our first goal is to define quasi-L\'evy measures of quasi-infinitely divisible distributions.
We can basically view them as the difference of the L\'evy measures $\nu_1$ and $\nu_2$ of two infinitely divisible
distributions $\mu_1$ and $\mu_2$. However, the difference is not a signed measure if $\nu_1$ and $\nu_2$ are infinite; on the other hand, when $\nu_1$ and $\nu_2$ are restricted to $\bR \setminus (-r,r)$ for some $r>0$, then the difference is a finite signed measure. Hence we can formalise the following definition.

\begin{defn} \label{def-LM1}
Let $\cB_r := \{ B \in \cB: B \cap (-r,r) = \emptyset\}$ for $r>0$ and
$\cB_0 = \bigcup_{r>0} \cB_r$  be the class of all Borel sets that are bounded away from zero. Let $\nu : \cB_0 \to \bR$ be a function such that $\nu_{|\cB_r}$ is a finite signed measure for each $r>0$, and denote the total variation, positive and negative part of $\nu_{|\cB_r}$ by $|\nu_{|\cB_r}|$, $\nu^+_{|\cB_r}$ and $\nu^-_{|\cB_r}$, respectively. Then the  \emph{total variation} $|\nu|$, the \emph{positive part $\nu^+$} and the \emph{negative part $\nu^-$} of $\nu$ are defined to be the unique measures on $(\bR,\cB)$ satisfying
$$|\nu|(\{0\}) = \nu^+ (\{0\})= \nu^-(\{0\}) = 0$$
and $$|\nu|(A) = |\nu_{|\cB_r}|(A), \;\nu^{+}(A) = (\nu_{|\cB_r})^{+}(A), \; \nu^{-}(A) = (\nu_{|\cB_r})^{-}(A)$$
for $A\in \cB_r$ for some $r>0$.
\end{defn}

Observe that when $\nu\cl\cB_0 \to \bR$ is such that $\nu_{|\cB_r}$ is a finite signed measure for each $r>0$, then $|\nu_{|\cB_r}|(A) = |\nu_{|\cB_s}|(A)$ for all $A \in \cB_r$ with $0<s\leq r$ and similarly for the positive and negative parts, so that $|\nu|$, $\nu^+$ and $\nu^-$ are well-defined and it is easy to see that these measures on $(\bR,\cB)$ indeed exist and are necessarily unique.   Observe that $\nu$ itself is defined on $\cB_0$, which is not a $\sg$-algebra, hence $\nu$ is not a signed measure. It is not always possible to extend the definition of $\nu$ to $\cB$ such that $\nu$ will be a signed measure. However, whenever it is possible, we will identify $\nu$ with its extension to $\cB$ and speak of $\nu$ as a signed measure. Then $\nu(\{0\}) = 0$ and the total variation, positive and negative parts of $\nu$ as defined in Definition \ref{def-LM1} coincide with the corresponding notions from \eqref{eq-variation} and \eqref{eq-Hahn} for the signed measure $\nu$.

We can now define quasi-L\'evy measures and quasi-L\'evy type measures:

\begin{defn} \label{def-LM2}
(a) A \emph{quasi-L\'evy type measure} is a function $\nu:\cB_0 \to \bR$ satisfying the condition in Definition \ref{def-LM1} such that its total variation $|\nu|$ satisfies $\int_\bR (1 \wedge x^2) \, |\nu|(\di x) < \infty$.\\
(b)
Suppose that $\mu$ is a quasi-infinitely divisible distribution on
$\R$ with characteristic pair $(\zt,\gm)$ with respect to a representation function $c$.  Then $\nu\cl \mcal B_0 \to \R$
defined by
\begin{equation}\label{eq-qLm}
\nu(B) = \int_B (1\land x^2)^{-1} \zt(\di x),\qq B\in\mcal B_0
\end{equation}
is called the {\it quasi-L\'evy measure} of $\mu$.
\end{defn}

For the quasi-L\'evy measure $\nu$ of a quasi-infinitely divisible
distribution $\mu$, we have $\int_{\R} (1\land x^2) |\nu| (\di x)<\infty$, where $|\nu|$ is the total
variation of $\nu$.
Hence every quasi-L\'evy measure of some distribution is also a quasi-L\'evy type measure, but the converse is not true, as will be seen in Example \ref{ex-I}. Observe that the notion \lq\lq quasi-L\'evy measure\rq\rq~ is used only when a quasi-infinitely divisible distribution is described, while the notion \lq\lq quasi-L\'evy type measure\rq\rq~ is not necessarily related to a distribution.

We say that a function $f\cl \bR\to \bR$ is \emph{integrable with respect to a quasi-L\'evy type measure $\nu$}, if it is integrable with respect to $|\nu|$ (hence also with respect to $\nu^+$ and $\nu^-$), and we then define
$$
\int_B f(x) \, \nu(\di x) := \int_B f(x) \, \nu^+(\di x) - \int_B f(x) \, \nu^-(\di x),\qq B\in\mcal B,
$$
although $\nu$ is not always a signed measure on $\R$.
For a representation function $c$, the function $x\mapsto \re^{\ri z x} - 1 - \ri z c(x)$ is integrable with respect to $\nu$.
Now we can speak of characteristic triplets:

\begin{defn} \label{def-LM3}
Let $\mu$ be a quasi-infinitely divisible distribution with characteristic pair $(\zeta,\gamma)$ with respect to
$c$. Then $(a,\nu,\gamma)$, where $a:= \zeta(\{0\})$ and $\nu$ is the quasi-L\'evy measure of $\mu$ defined by Definition \ref{def-LM2}~(b),
is called the \emph{characteristic triplet} of $\mu$ with respect to $c$. It is necessarily unique and $\zt$ is uniquely restored from $a$ and $\nu$.
We write $\mu \sim \qid (\zeta,\gamma)_c$ and $\mu \sim \qid (a,\nu,\gamma)_c$ to indicate that $\mu$ is quasi-infinitely divisible with given characteristic pair or triplet. The constant $a$ is called the \emph{Gaussian variance} of $\mu$.
\end{defn}

Notice that
\begin{equation}\label{eq-zeta}
\zt(B) = a\dl_0(B)+\int_B (1\land x^2) \nu(\di x),\qq B\in\mcal B.
\end{equation}
The characteristic function of a quasi-infinitely divisible distribution $\mu$ satisfies \eqref{eq-trunc2}
where  the characteristic exponent $\Psi_\mu$ of $\mu$ is given by \eqref{eq-trunc3}
with $a,\gamma\in \bR$ and $\nu$ being the quasi-L\'evy measure of $\mu$. The characteristic function of
a quasi-infinitely divisible distribution obviously cannot have zeroes.

\begin{rem}
As is explained in Section 1, $\mu$ is a quasi-infinitely divisible distribution on $\R$ if and only if
there are two infinitely divisible distributions $\mu_1$, $\mu_2$ such that $\wh\mu(z)=\wh\mu_1(z) / \wh\mu_2 (z)$.
We can define quasi-infinitely divisible distributions on $\R^d$ by this property.
Alternatively, Definition \ref{def-LM1} can be extended to $\R^d$ word by word with $\mcal B$ defined as
the class of all Borel sets in $\R^d$ and $\mcal B_r$ as the class $\{ B\in\mcal B\cl B\cap \{x\cl |x|<r\}
=\emptyset \}$.
A \emph{quasi-infinitely divisible distribution on $\bR^d$} can then be defined as a distribution $\mu$ on $\bR^d$ whose characteristic function $\widehat{\mu}(z) = \int_{\bR^d} \re^{\ri \langle z, x\rangle} \, \mu(\di x)$ admits a representation
$$\widehat{\mu}(z) = \exp \left(  \ri \langle \gamma, z \rangle - \frac12 \langle Az , z\rangle + \int_{\bR^d} \left(
\re^{\ri \langle z,x\rangle } - 1 - \ri \langle z, c(x)\rangle \right) \, \nu(\di x) \right), \quad z \in \bR^d,$$
for a fixed representation function $c$, where $\gamma \in \bR^d$, $A$ is a symmetric $d\times d$-matrix, and $\nu$ is a function $\mcal B_0 \to \R$ such that
$\nu_{|\cB_r}$ is a finite signed measure for each $r>0$ and $\int_{\R^d} (1\land |x|^2) |\nu|(\di x)<\infty$.
Here, $\langle z, x \rangle$ denotes the standard inner product of $z,x\in\bR^d$, and by a representation function we mean a bounded,
Borel measurable function $c\cl \bR^d \to \bR^d$ such that $|x|^{-2} |c(x)-x|\to 0$ as $x\to 0$ in $\R^d$.
It is possible to show that $(A,\nu,\gamma)$ is unique (cf. Sato \cite[Exercise 12.2]{Sa}) and can hence be called the characteristic triplet of $\mu$. In this paper, mainly for simplicity, we shall restrict ourselves to the one-dimensional case.
\end{rem}

For the expression of the characteristic functions of quasi-infinitely divisible distributions, it is possible to
replace representation functions by other functions as long as the corresponding integral is defined.
This is similar to the case of infinitely divisible distributions.  Particular important replacement is by $0$ or $x$:

\begin{rem} \label{rem-drift}
Let $\mu \sim \qid (a,\nu,\gamma)_c$ for some $c$, where $\nu$ is such that $\int_{|x|<1} |x| \, |\nu|(\di x) < \infty$.
Then $\re^{\ri z x} - 1$ is integrable with respect to $\nu$, and $\widehat{\mu}$ can be represented as
$$\widehat{\mu}(z) = \exp \left( \ri \gamma_0 z - a z^2/2 +  \int_\bR (\re^{\ri z x} - 1 ) \, \nu(\di x) \right)$$
for some $\gamma_0 \in \bR$; more precisely, $\gamma_0 = \gamma - \int_\bR c(x) \, \nu(\di x)$. This representation is unique and $\gamma_0$ is called the \emph{drift} of $\mu$.
We also write $\mu\sim \qid (a,\nu,\gamma_0)_0$ or $\mu \sim \qid (\zeta,\gamma_0)_0$ to indicate that $\int_{|x|<1} |x| \, |\nu|(\di x) < \infty$ and that $\mu$ has drift $\gm_0$.\\
Similarly, if $\nu$ is such that $\int_{|x|>1}  |x| \, |\nu|(\di x) < \infty$, then $\re^{\ri z x} - 1 - \ri z x$ is integrable with respect to $\nu$, and $\widehat{\mu}$ can be represented as
$$\widehat{\mu}(z) = \exp \left( \ri \gamma_m z - a z^2/2 + \int_\bR (\re^{\ri z x} - 1 - \ri z x) \, \nu(\di x) \right)$$
for some $\gamma_m \in \bR$. The representation is unique and $\gamma_m$ is called the \emph{center} of $\mu$ and related to $\gamma$ by $\gamma_m = \gamma + \int_{\bR}  (x-c(x)) \nu(\di x)$. We shall see in Theorem \ref{thm-moments1} that the center of a quasi-infinitely divisible distribution is equal to its mean.
\end{rem}

\begin{rem} \label{rem-conv}
(a) The class of quasi-infinitely divisible distributions is closed under convolution. More precisely, if $\mu_1 \sim \qid (a_1,\nu_1,\gamma_1)_c \sim \qid (\zeta_1,\gamma_1)_c$ and $\mu_2 \sim \qid (a_2,\nu_2,\gamma_2)_c \sim \qid (\zeta_2,\gamma_2)_c$, then $\mu_1\ast \mu_2 \sim \qid (a_1+a_2, \nu_1+\nu_2,\gamma_1+\gamma_2) \sim \qid (\zeta_1+\zeta_2,\gamma_1+\gamma_2)_c$. Similarly, the drift or center of convolutions is the sum of the individual drifts or centers, provided they exist.\\
(b) The class of quasi-infinitely divisible distributions is also closed under shifts and dilation, i.e. if $\mu= \law(X)$ for some random variable $X$ is quasi-infinitely divisible, then also $\law(mX+b)$ is quasi-infinitely divisible for $m,b\in \bR$ with $m\neq 0$. More precisely, if $\law(X) \sim \qid (a,\nu,\gamma)_c$, then $\law(mX+b) \sim \qid (a m^2, \overline{\nu}, b+m\gamma + \int_\bR (c(mx) - m c(x)) \, \nu(\di x))_c$ with $\overline{\nu}(B) := \nu(m^{-1} B)$, $B\in \cB$, as can be easily seen by considering the characteristic function of $mX+b$. Similarly, if $\law(X)$ has finite drift $\gm_0$ or center $\gamma_m$, then also $mX+b$ has finite drift given by $m\gm_0 +b$, or center given by $m\gamma_m + b$, respectively.\\
(c) We have already seen that not every pair $(\zeta,\gamma)_c$ with $\zeta$ being a finite signed measure
gives rise to a quasi-infinitely divisible distribution via \eqref{eq-trunc5}.  Similarly, not every triplet $(a,\nu,\gamma)_c$ with
$\nu$ being a quasi-L\'evy type measure
gives rise to a quasi-infinitely divisible distribution via \eqref{eq-trunc3}.
Of course $\gm$ is irrelevant to this property, which follows from (b).
We can say that, if $(\zeta,\gamma)_c$ is the characteristic pair of a quasi-infinitely divisible distribution $\mu$, then so is $(\zeta',\gamma')_c$ for some distribution $\mu'$ whenever $\gamma'\in \bR$ and $\zeta'\geq \zeta$ in the sense that $\zeta'(B) \geq \zeta(B)$ for all $B\in \cB$; similarly, if $(a,\nu,\gamma)_c$ is the characteristic triplet of a quasi-infinitely divisible distribution $\mu$, then so is $(a',\nu',\gamma')_c$  for $\mu'$ whenever $\gamma'\in \bR$, $a'\geq a$ and $\nu'\geq \nu$ in the sense that $\nu'(B) \geq \nu(B)$ for all $B\in \cB_0$. This is seen by letting $\mu''$ be
an infinitely divisible distribution with characteristic pair $(\zeta'-\zeta,\gamma'-\gamma)_c$, or characteristic triplet $(a'-a,\nu'-\nu,\gamma'-\gamma)_c$, respectively, and observing that $\mu' = \mu \ast \mu''$.
\end{rem}

We allowed also negative Gaussian variances $a= \zeta(\{0\})$ in the definition of quasi-infinitely divisible distributions. The next lemma shows that necessarily $a\geq 0$.

\begin{lem} \label{lem-a}
Let $\mu \sim \qid (a,\nu,\gamma)_c \sim \qid (\zeta,\gamma)_c$ for some $c$. Then
$$a = \zeta(\{0\}) = -2 \lim_{|z|\to \infty} z^{-2} \Psi_\mu(z).$$
In particular, $a\geq 0$.
\end{lem}

\begin{proof}
We have
\begin{eqnarray*}
& & \lim_{|z|\to \infty} z^{-2} \int_\bR \left( \re^{\ri z x} - 1 - \ri z c(x) \right) \, \nu(\di x) \\
& = & \lim_{|z|\to \infty} z^{-2} \left(  \int_\bR \left( \re^{\ri z x} - 1 - \ri z c(x) \right) \nu^+(\di x) -
 \int_\bR \left( \re^{\ri z x} - 1 - \ri z c(x) \right) \nu^-(\di x) \right)= 0
\end{eqnarray*} by Sato \cite[Lem.\ 43.11]{Sa}. Hence $\lim_{|z|\to \infty} z^{-2} \Psi_\mu(z) = -a/2$. Now if $a$ were strictly negative,
then $|\widehat{\mu}(z)| = |\exp (\Psi_\mu(z))|$ would tend to $\infty$ as $|z| \to \infty$, which is clearly impossible for a characteristic function. Hence $a\geq 0$.
\end{proof}

As seen in \eqref{eq-measure-estim3} below, if the Gaussian variance in a quasi-infinitely divisible distribution is zero, then the positive part of the quasi-L\'evy measure must
have at least as much mass as the negative part. More generally, we have:

\begin{lem}\label{lem-de2}
Let $\mu \sim \qid (a,\nu,\gamma)_c$ for some $c$.  Let $\sigma$ be an arbitrary probability distribution on $\bR$. Then
\begin{align}
\frac{a}{2} z^2 +\int_{\bR}(1-\cos zx)\nu^+(\di x) & \geq \int_{\bR}(1-\cos zx)\nu^-(\di x), \quad  \forall z\in \bR, \label{eq-measure-estim1}\\
\frac{a}{2}\int_{\bR} z^2 \sg(\di z)+\int_\bR (1 - \Re \widehat{\sigma}(x)) \, \nu^+ (\di x) & \geq \int_\bR (1 - \Re \widehat{\sigma}(x))\, \nu^-( \di x) , \quad \mbox{and} \label{eq-measure-estim2}\\
a + \int_\bR \frac{x^2}{1+x^2} \, \nu^+(\di x) & \geq \int_\bR \frac{x^2}{1+x^2} \, \nu^-(\di x). \label{eq-measure-estim4}
\end{align}
Further, if $a=0$, then
\begin{align} \label{eq-measure-estim3}
\nu^+ (\bR) & \geq  \nu^-(\bR), \quad \mbox{and} \\
\int_{\bR} (1\land |x|) \nu^+ (\di x)  & \geq \int_{\bR} (1\land |x|) \nu^- (\di x) \label{eq-measure-estim5}
\end{align}
In particular, if $a=0$ and $\nu^+(\bR)$ is finite, then so is $\nu^-(\bR)$.
\end{lem}

\begin{proof}
Assertion \eqref{eq-measure-estim1} follows from  $$0  \geq \log |\widehat{\mu}(z)| = \Re (\Psi_\mu(z)) =
  -\frac{a}{2} z^2 + \int_{\bR}(\cos zx -1)\nu^+(\di x)-\int_{\bR}(\cos zx -1 )\nu^-(\di x).$$ Hence
  $$\frac{a}{2} \int_\bR z^2 \, \sigma(\di z) + \int_\bR \int_\bR (1-\cos zx) \, \nu^+(\di x) \, \sigma ( \di z) \geq
  \int_\bR \int_\bR (1-\cos zx) \, \nu^-(\di x) \, \sigma ( \di z),$$
and an application of Fubini's theorem gives \eqref{eq-measure-estim2}. Assertion \eqref{eq-measure-estim4} follows from \eqref{eq-measure-estim2} by choosing $\sigma$ as the two-sided exponential distribution $\sigma (\di x) = 2^{-1} \re^{-|x|}\, \di x$ for which $\widehat{\sigma}(x) = 1/(1+x^2)$ and $\int_\bR z^2 \, \sigma(\di z) = 2$.
Now let $a=0$ and $\sigma$ be an $N(0,t)$ distribution with $t>0$. Then $\widehat{\sigma}(x) = \re^{-tx^2/2}$, and letting $t\to \infty$ in \eqref{eq-measure-estim2} (with $\sigma = N(0,t)$) gives \eqref{eq-measure-estim3} by monotone convergence.
To see \eqref{eq-measure-estim5}, letting $\sg(\di x)=\pi^{-1} x^{-2} (1-\cos x) \di x$, we have $\wh\sg (x)= (1-|x|)\lor 0$
and $1-\Re \wh \sg (x) = 1\land |x|$; see \cite[p.503]{Feller2}.
\end{proof}

Lemma \ref{lem-de2} can be used to show that certain triplets do not lead to characteristic triplets of quasi-infinitely divisible
distributions via \eqref{eq-trunc3}. For example, $(0,\delta_1 - 2 \delta_3,\gamma)_c$ is not the characteristic triplet of a quasi-infinitely divisible distribution, since \eqref{eq-measure-estim3} is violated. Deeper results of this kind can be obtained using the class $I_0$:

\begin{ex} \label{ex-I}
\emph{The class $I_0$} consists of all infinitely divisible probability distributions $\mu$ such that each factor of $\mu$ is also infinitely divisible, i.e. such that $\mu = \mu_1 \ast \mu_2$ with probability distributions $\mu_1$ and $\mu_2$ implies infinite divisibility of $\mu_1$ and $\mu_2$; by Khintchine's theorem (e.g. \cite[Thm.~5.4.2]{Linnik} or \cite[Thm.~4.6.1]{Cuppens1975}) this is equivalent to the more common definition that a probability distribution belongs to $I_0$ if every factor of it is decomposable. Now if $\mu$ is in $I_0$ with characteristic triplet $(a,\nu,\gamma)_c$ and if $(a',\nu',\gamma')_c$ is given with $a',\gamma'\in \bR$ such that $a'\leq a$, $(\nu')^- \neq 0$ and $\nu'\leq \nu$ in the sense that $\nu'(B) \leq \nu(B)$ for all $B\in \cB_0$, then (as also noted in Cuppens \cite[Cor.~4.6.2]{Cuppens1975}) $(a',\nu',\gamma')_c$ is not the characteristic triplet of a quasi-infinitely divisible distribution. For suppose it were, and denote it by $\mu'$. Let $\mu''$ be the infinitely divisible distribution with characteristic triplet $(a-a',\nu-\nu', \gamma-\gamma')_c$. Then $\mu' \ast \mu'' = \mu$ with $\mu'$ not being infinitely divisible, contradicting $\mu \in I_0$.
Sufficient conditions for a distribution to be in $I_0$ can be found e.g.\ in Linnik \cite{Linnik} or Cuppens \cite{Cuppens1975}.
For example, Gaussian distribution, Poisson distribution, and
the convolution of a Gaussian and a Poisson are in $I_0$ (\cite[Thms.~6.3.1, 6.6.1, 7.1.1]{Linnik}).
Hence, if $\nu^- \ne 0$ and if either $\nu^+=0$ or $\supp \nu^+$ is a one-point set, then $(a,\nu,\gm)_c$ is not the characteristic triplet
of a quasi-infinitely divisible distribution for any $a$ and $\gm$; in other words, in this case $\nu$ is a quasi-L\'evy type measure but there is no distribution $\mu$ for which $\nu$ will be the quasi-L\'evy measure.
An infinitely divisible distribution with Gaussian variance 0 and L\'evy measure of the form $\nu= \sum_{k=1}^n b_k \delta_{\tau_k}$, where $0 < \tau_1 < \ldots < \tau_n$, $b_1,\ldots, b_n > 0$ and either $\tau_1,\ldots, \tau_n$ are linearly independent over $\mathbb{Q}$, or $\tau_n \leq 2 \tau_1$, belongs to $I_0$ by results of Raikov as stated in \cite[Thms.~12.3.2 and 12.3.3]{Linnik}. More generally, if an infinitely divisible distribution has Gaussian variance 0 and L\'evy measure $\nu$ such that $\supp \nu \subset (b,2b)$ for some $b>0$, then it belongs to $I_0$,
cf.\ \cite[Cor.~7.1.1]{Cuppens1975}. Further examples of distributions in $I_0$ are given in \cite[Thms.~9.0.1 and 10.0.1]{Linnik}, \cite{LinnikOstrovskii} or \cite{Cuppens1975}.
\end{ex}


\section{Examples} \label{S3}

Obviously, every infinitely divisible distribution on $\bR$ is quasi-infinitely divisible, and its L\'evy measure and quasi-L\'evy
measure coincide. An important example of quasi-infinitely divisible distributions has been established by Cuppens \cite{Cuppens1969}. Namely, a distribution which has an atom of mass $> 1/2$ is necessarily quasi-infinitely divisible. More precisely, it holds:

\begin{thm} {\rm (Cuppens \cite[Prop.~1]{Cuppens1969}, \cite[Thm.~4.3.7]{Cuppens1975})} \label{thm-cuppens}
Let $\mu$ be a non-degenerate distribution such that there is $\lambda\in \bR$ with $p= \mu (\{\lambda\})  > 1/2$ and define
$\sigma= (1-p)^{-1} (\mu - p \delta_\lambda)$. Then $\mu$ is quasi-infinitely divisible with finite quasi-L\'evy measure $\nu$ given by
$$\nu = \left( \sum_{m=1}^\infty \frac{1}{m} (-1)^{m+1} \left(
\frac{1-p}{p} \right)^m (\delta_{-\lambda} \ast \sigma)^{\ast m}
\right)_{| \mathbb{R} \setminus \{ 0 \}}$$
drift $\lambda$, and Gaussian part $a=0$, i.e. its
characteristic function admits the representation
$$\widehat{\mu}(z) = \exp \left( i \lambda z + \int_{\bR} (e^{izx}
-1) \, \nu(\di x) \right), \quad z\in \bR.$$
\end{thm}

Theorem \ref{thm-cuppens} gives rise to many examples of quasi-infinitely divisible distributions that are not infinitely divisible. In particular, if $\mu$ has an atom of mass in $(1/2,1)$ and has bounded support, then it is quasi-infinitely divisible without being infinitely divisible, since the only infinitely divisible distributions with bounded support are the Dirac measures (cf. \cite[Cor. 24.4]{Sa}). Since convolutions of quasi-infinitely divisible distributions are also quasi-infinitely divisible, this allows to detect also quasi-infinitely divisible distributions that have atoms with masses less than $1/2$.

\begin{ex} \label{ex-bin}
It follows from Theorem \ref{thm-cuppens} that a two-point distribution $p\delta_\lambda + (1-p) \delta_{\lambda'}$ is quasi-infinitely divisible as long as $p\neq 1/2$. In particular, the Bernoulli distribution $b(1,p)$ is quasi-infinitely divisible for $p\neq 1/2$. Since convolutions of quasi-infinitely divisible distributions are quasi-infinitely divisible, also the binomial distribution $b(n,p)$ with parameters $n\in \bN$ and $p\in (0,1)$ is quasi-infinitely divisible as long as $p\neq 1/2$.
\end{ex}

The characteristic function of a  quasi-infinitely distribution cannot have zeroes. Hence, a two-point distribution of the form $\mu = (1/2) \delta_{\lambda} + (1/2)\delta_{\lambda'}$ with $\lambda \neq \lambda'$ is not quasi-infinitely divisible. Also, the characteristic function of the $b(n,1/2)$-distribution has zeroes, so $b(n,1/2)$ is not quasi-infinitely divisible. In particular, for $n\in \bN$ and $p\in (0,1)$ we see that $b(n,p)$ is quasi-infinitely divisible if and only if its characteristic function has no zeroes, and if and only if $p\neq 1/2$.

It is natural to ask if every distribution whose characteristic function does not have zeroes must be quasi-infinitely divisible. The following example shows that this is not the case:

\begin{ex} \label{ex-linnik}
Let $\varphi\cl \bR \to \bR$ be defined by
$$\varphi(z) = \begin{cases} (1/ 7) \exp (1-z^4), & |z|\geq 1,\\
(2/7) z^2 - (8/7) |z| + 1 , & |z| < 1.
\end{cases}$$
Then $\varphi$ is a real-valued, even  and continuous function with $\varphi(0) = 1$ and
$\lim_{z\to \infty} \varphi(z) = 0$. Further, $\varphi$ is $C^2$ on $(0,\infty)$ with
strictly positive second derivative there, hence $\varphi$ is convex on $(0,\infty)$.
It follows from P\'olya's theorem (e.g.\ Lukacs \cite[Thm.\ 4.3.1]{L} or Feller
\cite[XV.3, Ex.\ (b)]{Feller2}) that $\varphi$ is the characteristic function of an
absolutely continuous distribution, $\mu$ say. Observe that $\varphi(z) \neq 0$ for all
$z\in \bR$, but that $\lim_{z\to\infty} z^{-2} \log \varphi(z) = -\infty$. Hence $\mu$
is not quasi-infinitely divisible by Lemma \ref{lem-a}, although its characteristic
function has no zeroes.
\end{ex}

We have seen that not every probability measure whose characteristic function is non-vanishing is quasi-infinitely divisible. However, for distributions concentrated on the integers, this does not happen, as we shall see in Section \ref{S-integers}. In this section in Theorem \ref{thm-finite} we will prove a special case of this result for distributions concentrated on $\{0,1,\ldots, n\}$; this is more elementary, the quasi-L\'evy measure can be given more explicitly, and the special case will be needed in the proof of the general result in Theorem \ref{thm-integers}.

For the proof of Theorem \ref{thm-finite}, we will need a generalisation of Cuppens' Theorem stated above, which we do now for
complex-valued measures rather than probability distributions; this will be helpful later when factorizing the characteristic function of a probability distribution on $\{0,1,\ldots, n\}$. Recall that a \emph{complex measure} $\rho$ on $\bR$ is a function $\rho:\cB \to \bC$ such that $\rho(\emptyset) = 0$ and $\rho (\bigcup_{j=1}^\infty A_j) = \sum_{j=1}^\infty \rho (A_j)$ for all sequences $(A_j)_{j\in \bN}$ of pairwise disjoint sets in $\cB$; this implies that the series converges unconditionally, in particular absolutely for each partition. A complex measure is automatically finite. Its total variation $|\rho|$ is defined by formula \eqref{eq-variation}. This is a finite measure. The Fourier transform of a complex measure $\rho$ is defined by $\widehat{\rho}(z) = \int_\bR \re^{\ri z x} \, \rho(\di x)$. It satisfies $|\widehat{\rho}(z)|\leq |\rho|(\bR)$ for all $z\in \bR$. We come now to the aforementioned generalisation of Cuppens' result:

\begin{prop} \label{prop-cuppens2}
Let $\alpha$ and $\beta$ be two complex measures such that $\widehat{\alpha}(z) \neq 0$ for all $z\in \bR$. Suppose there is a complex measure $\rho$ with $|\rho|(\bR) < 1$ such that $\widehat{\rho}(z) = \widehat{\beta}(z) /\widehat{\alpha}(z)$ for all $z\in \bR$. Define the complex measure $\widetilde{\nu}$ by
$$\widetilde{\nu} = \left( \sum_{m=1}^\infty \frac{1}{m} (-1)^{m+1} \rho^{\ast m} \right)_{|\bR \setminus \{0\}}.$$
Then
$$(\alpha +\beta)\,{\widehat{ }} \,(z) = \frac{\alpha(\bR) + \beta(\bR)}{\alpha(\bR)}\, \widehat{\alpha}(z) \, \exp \left( \int_\bR (\re^{\ri x z} - 1) \, \widetilde{\nu}(\di x) \right), \quad z \in \bR.$$
\end{prop}

\begin{proof}
First observe that
$$(\alpha +\beta)\,{\widehat{ }} \,(z) = \widehat{\alpha}(z) \left( 1 + \frac{\widehat{\beta}(z)}{\widehat{\alpha}(z)} \right) = \widehat{\alpha}(z) \exp \left( \log (1+ \widehat{\rho}(z) ) \right), \quad z\in \bR.$$
Since $|\widehat{\rho}(z)|\leq |\rho| < 1$ we can use the logarithmic expansion $\log (1+w) =\break
\sum_{m=1}^\infty (-1)^{m+1} m^{-1} w^m$ for $|w|<1$ and continue
\begin{eqnarray*}
\log ( 1+\widehat{\rho}(z)) & = & \sum_{m=1}^\infty (-1)^{m+1} m^{-1} (\widehat{\rho}(z))^m \\
& = & \left( \sum_{m=1}^\infty (-1)^{m+1} m^{-1} \rho^{\ast m} \right)^{\wedge} (z) \\
& = & \int_\bR (\re^{\ri z x} - 1) \left( \sum_{m=1}^\infty (-1)^{m+1} m^{-1} \rho^{\ast m} \right) (\di x) + \sum_{m=1}^\infty (-1)^{m+1} m^{-1} \rho^{\ast m} (\bR) \\
& = & \int_\bR (\re^{\ri z x} - 1) \widetilde{\nu}(\di x) + \log (1+\rho(\bR)),
\end{eqnarray*}
where in the last line we used that $(\re^{\ri z x} - 1)_{|x=0} = 0$, so that a point mass of the measure at 0 is ignored in the integration. Since $$\exp \left( \log (1+\rho(\bR)) \right) = 1 + \widehat{\rho}(0) = \frac{\widehat{\alpha}(0) + \widehat{\beta}(0)}{\widehat{\alpha}(0)} = \frac{\alpha(\bR) + \beta(\bR)}{\alpha(\bR)}$$
this gives the claim.
\end{proof}

The above result can in particular be applied to convex-combinations of probability measures:

\begin{cor} \label{cor-cuppens3}
Let $p>q> 0$ with $p+q=1$  and $\mu_1$ and $\mu_2$ be two probability distributions on $\bR$ such that
$\mu_1$ is quasi-infinitely divisible with characteristic triplet $(a,\nu,\gamma)_c$ with respect to $c$.
Suppose further that there exists a finite signed measure $\sigma$ on $\bR$ with $|\sigma|(\bR) < p/q$ and
$\widehat{\sigma}(z) = \widehat{\mu_2}(z) / \widehat{\mu_1}(z)$ for all $z\in \bR$. Define a finite signed measure $\widetilde{\nu}$ by
$$\widetilde{\nu} = \left( \sum_{m=1}^\infty \frac{1}{m} (-1)^{m+1} (q/p)^m \sigma^{*m} \right)_{|\bR \setminus \{0\}}.$$
Then $p \mu_1 + q \mu_2$ is quasi-infinitely divisible with characteristic triplet $(a,\nu+\widetilde{\nu}, \gamma+\int_\bR c(x) \widetilde{\nu}(\di x))_c$. If additionally $\int_{|x|<1} |x| \, |\nu|(\di x) < \infty$ and $\mu_1$ has drift $\lambda$, then also $p\mu_1 + q \mu_2$ has drift $\lambda$, i.e. it has characteristic triplet $(a,\nu+\widetilde{\nu},\lambda)_0$.
\end{cor}

\begin{proof}
Define $\alpha = p \mu_1$, $\beta= q \mu_2$ and $\rho = (q/p) \sigma$.
Since $\mu_1$ is quasi-infinitely divisible, we have $\widehat{\mu_1}(z) \neq 0$ and $\widehat{\rho}(z) = (q/p)\widehat{\sigma}(z)
= \widehat{\beta}(z) / \widehat{\alpha}(z)$. By Proposition \ref{prop-cuppens2} we then obtain
\begin{align*}
&(p \mu_1 + q \mu_2)\,\widehat{ { }}\;\,(z) = \widehat{\alpha}(z) + \widehat{\beta}(z)
= \frac{p+q}{p} p \widehat{\mu_1}(z) \, \exp \left( \int_\bR (\re^{\ri z x} - 1) \, \widetilde{\nu}(\di x) \right) \\
&\qq = \exp \left( -az^2/2 + \int_\bR (\re^{\ri z x} - 1 - \ri z c(x) ) \, \nu(\di x) +\ri \gm z + \int_\bR (\re^{\ri z x}
- 1) \, \widetilde{\nu}(\di x) \right).
\end{align*}
This shows that $p\mu_1 + q \mu_2$ is quasi-infinitely divisible with characteristic triplet $(a,\nu + \widetilde{\nu}, \gamma + \int_\bR c(x) \, \widetilde{\nu}(\di x))_c$. The drift assertion follows in the same way.
\end{proof}

Corollary \ref{cor-cuppens3} contains Cuppens' result (Theorem \ref{thm-cuppens}) as a special case.
To see this, let $\mu$ be a non-degenerate distribution that has an atom of mass $p= \mu(\{\lambda\}) > 1/2$ at $\lambda$.
Define $\mu_1= \delta_\lambda$ and $\mu_2= (1-p)^{-1} (\mu - p \delta_\lambda)$. Then $\mu_1$ is infinitely divisible and
$$\frac{\widehat{\mu_2}(z)}{\widehat{\mu_1}(z)} = \widehat{\mu_2}(z) \, \widehat{\delta_{-\lambda}}(z) = (\mu_2 \ast \delta_{-\lambda})\, \widehat{ { }}\;\, (z).$$
Theorem \ref{thm-cuppens} then follows from Corollary \ref{cor-cuppens3}.
Another example is the following:

\begin{ex} \label{ex-normal}
Let $b>a>0$, $\mu_1 = N(0,a)$, $\mu_2 = N(0,b)$, $p \in (1/2,1)$ and $q=1-p$. Define $\sigma = N(0,b-a)$. Then $\mu_1$ is infinitely divisible, and
$$\frac{\widehat{\mu_2}(z)}{\widehat{\mu_1}(z)} = \frac{\re^{-b z^2/2}}{\re^{-a z^2/2}} =  \widehat{\sigma}(z).$$
Corollary \ref{cor-cuppens3} then implies that $p \mu_1 + q \mu_2$ is quasi-infinitely divisible with characteristic triplet $(a,\widetilde{\nu},0)_0$ with $\widetilde{\nu}$ as given there.
Observe that $p\mu_1 + q \mu_2$ is a particular case of a variance mixture of normal distributions and, since the underlying
mixing distribution function has bounded support, it is known that $p \mu_1 + q \mu_2$ is not infinitely divisible, see
Kelker \cite[Thm.~2]{Kelker1971}.  Another proof that $p \mu_1 + q \mu_2$ is not infinitely divisible follows from
\cite[Rem.~26.3]{Sa}, since the tail of $p \mu_1 + q \mu_2$ is asymptotically equal to that of $q\mu_2$ but $p \mu_1 + q \mu_2$
is not Gaussian.
\end{ex}

The previous example can be generalised:

\begin{ex}
Let $\mu_1$ and $\mu_2$ be two quasi-infinitely divisible distributions with $\mu_1 \sim \qid (a_1, \nu_1, \gamma_1)_c$ and $\mu_2 \sim \qid (a_2,\nu_2,\gamma_2)_c$, where $0\leq a_1 \leq a_2$ and $\nu_1$ and $\nu_2$ are finite quasi-L\'evy measures such that $\nu_2-\nu_1$ is a positive measure ($\mu_1$ and $\mu_2$ could in particular be infinitely divisible). Then $p\mu_1 + (1-p) \mu_2$ is quasi-infinitely divisible for $p\in (1/2,1)$. This can be seen from the fact that
$$\frac{\widehat{\mu_2}(z)}{\widehat{\mu_1}(z)} = \exp \left( \ri (\gamma_2 - \gamma_1) z - (a_2-a_1)z^2/2 + \int_{\bR} (\re^{\ri z x} - 1 - \ri z c(x)) \, (\nu_2-\nu_1)(\di x) \right),$$
which is the characteristic function of an infinitely divisible distribution $\sigma$, and hence Corollary \ref{cor-cuppens3} applies.
\end{ex}

The following lemma exploits Proposition \ref{prop-cuppens2} in more detail, and will be needed in the proof of Theorem \ref{thm-finite}.

\begin{lem} \label{lem-finite}
Let $\xi\in \bC$ with $|\xi|\neq 1$. Then the characteristic function of the complex measure $\mu = \delta_1 - \xi \delta_0$ satisfies
$$\widehat{\mu}(z) = \begin{cases} (1-\xi) \exp \left( \ri z + \int_\bR (\re^{\ri z x} - 1) \left( -\sum_{m=1}^\infty m^{-1}  \xi^m \delta_{-m} \right)(\di x)\right), & \mbox{\rm if} \;\; |\xi| < 1, \\
(1-\xi) \exp \left( \int_\bR (\re^{\ri x z} - 1 ) \, \left( -\sum_{m=1}^\infty m^{-1} \xi^{-m} \delta_m \right) (\di x) \right), & \mbox{\rm if} \;\; |\xi| > 1.
\end{cases} $$
\end{lem}

\begin{proof}
Suppose first that $|\xi|<1$. Define $\alpha= \delta_1$, $\beta= -\xi \delta_0$ and $\rho= -\xi \delta_{-1}$. Then $\widehat{\alpha}(z) = \re^{\ri z} \neq 0$, $|\rho|(\bR) = |\xi|<1$ and $\widehat{\rho}(z) = -\xi \re^{-\ri z} = \widehat{\beta}(z) / \widehat{\alpha}(z)$. The claim then follows from Proposition \ref{prop-cuppens2}, by observing that $\rho^{\ast m} = (-1)^m \xi^m \delta_{-m}$.

Now suppose that $|\xi|>1$. Define $\alpha= -\xi \delta_0$, $\beta= \delta_1$ and $\rho = -\xi^{-1} \delta_1$. Again, $\widehat{\alpha}(z) = -\xi \neq 0$, $|\rho|(\bR) = |\xi^{-1}|<1$ and $\widehat{\rho}(z) = -\xi^{-1} \re^{\ri z} = \widehat{\beta}(z) / \widehat{\alpha}(z)$, and the claim follows from Proposition \ref{prop-cuppens2}, since $\rho^{\ast m}=(-1)^m \xi^{-m} \delta_m$.
\end{proof}

We can now characterise when a distribution concentrated on $\{0,1,\ldots, n\}$ is quasi-infinitely divisible.

\begin{thm} \label{thm-finite}
Let $\mu$ be a discrete distribution concentrated on $\{0,1,2, \ldots, n\}$ for some $n\in \bN$, i.e. $\mu = \sum_{j=0}^n a_j \delta_j$, where $a_0, \ldots, a_{n-1} \geq 0$, $a_n>0$ and $a_0 + \ldots + a_n = 1$. Then the following are equivalent:
\begin{enumerate}
\item[(i)] $\mu$ is quasi-infinitely divisible.
\item[(ii)] The characteristic function of $\mu$ has no zeroes.
\item[(iii)] The polynomial $w \mapsto \sum_{j=0}^n a_j w^j$ in the complex variable $w$ has no roots on the unit circle, i.e. $\sum_{j=0}^n a_j w^j \neq 0$ for all $w\in \bC$ with $|w|=1$.
\end{enumerate}
Further, if one of the equivalent conditions (i) -- (iii) holds, then the quasi-L\'evy measure of $\mu$ is finite and concentrated on $\bZ$, the drift lies in $\{0,1\ldots, n\}$ and the Gaussian variance of $\mu$ is 0. More precisely, if $\xi_1,\ldots, \xi_n$ denote the $n$ complex roots of $w\mapsto \sum_{j=0}^n a_j w^j$, counted with multiplicity, then the quasi-L\'evy measure of $\mu$ is given by
\begin{equation} \label{eq-qlm1}
\nu = -  \sum_{m=1}^\infty m^{-1} \left( \sum_{j\cl |\xi_j|<1} \xi_j^m\right) \delta_{-m} -
\sum_{m=1}^\infty m^{-1} \left( \sum_{j\cl |\xi_j|>1} \xi_j^{-m} \right)\delta_m
\end{equation}
and the drift is equal to the number of zeroes of this polynomial that lie inside the unit circle (counted with multiplicity), i.e. have modulus less than $1$.
\end{thm}

\begin{proof}
Define the polynomial $f$ in $w$ by
$$f(w) := a_0 + a_1 w + \ldots + a_n w^n = a_n \left( w^n + \frac{a_{n-1}}{a_n} w^{n-1} + \ldots + \frac{a_1}{a_n} w + \frac{a_0}{a_n} \right).$$
Denoting by $\xi_1,\ldots, \xi_n$ the complex roots of $f$, counted with multiplicity, we can write
$$f(w) = a_n \prod_{j=1}^n (w-\xi_j).$$
The characteristic function of $\mu$ can be expressed as
\begin{equation} \label{eq-finite2}
\widehat{\mu}(z) = \sum_{j=0}^n a_j \re^{\ri j z } = f(\re^{\ri z}) = a_n \prod_{j=1}^n \left( \re^{\ri z} - \xi_j \right) = a_n \prod_{j=1}^n (\delta_1 - \xi_j \delta_0) \, \widehat{ { }}\;\, (z).
\end{equation}

Now assume that (iii) holds, i.e. that $|\xi_j|\neq 1$ for all $j\in \{1,\ldots, n\}$. Define the complex measure ${\nu}$ by
 \eqref{eq-qlm1}.
  Since $f$ has real coefficients, the non-real roots of $f$ appear as pairs of complex conjugates, from which it follows that ${\nu}$ is actually a finite signed measure. Denote by $\lambda$ the number of indices $j\in \{1,\ldots, n\}$ for which $|\xi_j| <1$. From Equation \eqref{eq-finite2} and Lemma \ref{lem-finite} we then obtain
$$\widehat{\mu}(z) = a_n \left(\prod_{j=1}^n (1-\xi_j)\right) \exp \left( \ri \lambda z + \int_\bR (\re^{\ri z x} - 1) \, \nu(\di x) \right),$$ which shows that $\mu$ is quasi-infinitely divisible with finite quasi-L\'evy measure $\nu$ and drift $\lambda$, since
$$a_n \prod_{j=1}^n (1-\xi_j) = f(1) = a_0 + \ldots + a_n = 1.$$ We have shown that (iii) implies (i) and given the specific form of the triplet. That (i) implies (ii) is obvious, and that (ii) implies (iii) can be seen from \eqref{eq-finite2}, since $\widehat{\mu}(z) \neq 0$ for all $z\in \bR$ implies $|\xi_j| \neq 1$ for all $j\in \{1,\ldots, n\}$.
\end{proof}

Later in Theorem \ref{thm-integers} we shall generalise Theorem \ref{thm-finite} and show that a distribution on the integers $\bZ$ is quasi-infinitely divisible if and only if its characteristic function has no zeroes. However, the proof of Theorem \ref{thm-integers} is on the one hand more complicated as it relies on a consequence of the Wiener-L\'evy theorem for absolutely summable Fourier series, and on the other hand also needs the assertion of Theorem \ref{thm-finite} in order to show that the derived candidate for quasi-L\'evy measure is indeed
real-valued.

A simple consequence of Theorem \ref{thm-finite} is the following:

\begin{cor} \label{cor-finite3}
Let $\mu$ be a discrete distribution concentrated on a finite subset of a  lattice of the form $r + h \bZ$ with $r\in \bR$ and $h>0$. Then $\mu$ is quasi-infinitely divisible if and only if its characteristic function has no zeroes. In this case, the quasi-L\'evy measure of $\mu$ is finite and the Gaussian variance is $0$.
\end{cor}

\begin{proof}
If the characteristic function of $\mu$ has zeroes it is clear that $\mu$ cannot be quasi-infinitely divisible.
Now suppose that $\widehat{\mu}$ has no zeroes.
Let $X$ be a random variable with distribution $\mu$. We then can find $k\in \bZ$ and $n\in \bN$ such that $Y= h^{-1} (X-r) +k$ is concentrated on $\{0, \ldots, n\}$. Then the characteristic function of $Y$ has no zeroes, hence $\law(Y)$ is quasi-infinitely divisible with Gaussian variance 0 and finite quasi-L\'evy measure by Theorem \ref{thm-finite}. The claim then follows from Remark \ref{rem-conv}~(b).
\end{proof}

So far, for all quasi-infinitely divisible distributions we encountered, the negative part $\nu^-$ of the quasi-L\'evy measure was finite. Next, we give an example of  quasi-infinitely divisible distributions with $\nu^-$ being infinite.

\begin{ex} \label{ex-inf-qid-measure}
Let $(X_k)_{k\in \bN}$ be an independent and identically distributed sequence of random variables with common distribution $(2/3) \delta_{-1} + (1/3) \delta_2$, and let $(b_k)_{k\in \bN}$ be a sequence of strictly positive real numbers such that $\sum_{k=1}^\infty b_k^2 < \infty$. Since the $X_k$ have expectation 0, the series $Y:= \sum_{k=1}^\infty b_k X_k$ converges almost surely (e.g. Feller \cite[Thm.\ VII.8.2]{Feller2}) and hence in distribution, regardless if $(b_k)_{k\in \bN}$ is summable or not. We claim that $Y$ is quasi-infinitely divisible with Gaussian variance 0, center 0 and quasi-L\'evy measure $\nu$ given by
\begin{equation} \label{eq-nu-infinite}
\nu = \sum_{k=1}^\infty \sum_{m=1}^\infty \frac{1}{m} (-1)^{m+1} 2^{-m} \delta_{3b_k m}.
\end{equation}
To see this, observe first that
$$\int_{\bR} (1 \wedge x^2) \, |\nu|(\di x) \leq \int_\bR x^2 \, |\nu|(\di x) \leq \sum_{k=1}^\infty \sum_{m=1}^\infty \frac{1}{m} 2^{-m} 9 b_k^2 m^2 < \infty.$$
Since $|\re^{\ri x z} - 1 - \ri x z| \leq x^2 z^2 /2$ and since
$$\int_\bR x^2 \sum_{k=n}^\infty \sum_{m=1}^\infty \frac{1}{m} 2^{-m} \delta_{3b_k m}(\di x) \leq 9
\sum_{k=n}^\infty b_k^2 \sum_{m=1}^\infty m 2^{-m} \to 0 \quad \mbox{as}\; n\to\infty,$$
it follows that for each $z\in \bR$,
\begin{eqnarray} \label{eq-nu-infinite2}
\lefteqn{\exp \left( \int_\bR \left( \re^{\ri x z} - 1 - \ri x z \right) \sum_{k=1}^n \sum_{m=1}^\infty m^{-1} (-1)^{m+1} 2^{-m} \delta_{3b_k m}(\di x) \right)} \\
\nonumber & \to & \exp \left( \int_\bR \left(\re^{\ri z x} - 1 - \ri x z \right)\, \nu (\di x)\right) \quad \mbox{as $n\to \infty$.}
\end{eqnarray}
By Theorem \ref{thm-cuppens}, $\law(b_k X_k) = (2/3) \delta_{-b_k} + (1/3) \delta_{2b_k}$ is quasi-infinitely divisible with Gaussian variance 0, quasi-L\'evy measure $\nu_{b_k} = \sum_{m=1}^\infty m^{-1} (-1)^{m+1} 2^{-m} \delta_{3b_k m}$ and drift $-b_k$. Since $\int_\bR x \, \nu_{b_k} (\di x) = b_k$, this implies that the center of $b_k X_k$ is 0 (alternatively, one can use Theorem \ref{thm-moments1} to be proved later). Hence the left-hand side of \eqref{eq-nu-infinite2} is the characteristic function of $\sum_{k=1}^n b_k X_k$. It follows that the right-hand side of \eqref{eq-nu-infinite2} is the characteristic function of $Y$, and that $Y$ is quasi-infinitely divisible with center 0, Gaussian variance 0 and quasi-L\'evy measure $\nu$.\\
Now suppose that the sequence $(b_k)_{k\in \bN}$ is additionally linearly independent over $\bQ$. Then there are no cancellations in the representation \eqref{eq-nu-infinite} of $\nu$ and $\nu^- = \sum_{k\in \bN} \sum_{m\in \bN, m \; {\rm even}} m^{-1} 2^{-m} \delta_{3b_k m}$. Then obviously $\nu^-(\bR) = \infty$ and $\int_0^{\infty} x^2 \nu^- (\di x)<\infty$.
For $\alpha \in (0,2]$ we have $\int_{(0,1)} x^\alpha \, \nu^-(\di x) < \infty$ if and only if $\sum_{k\in \bN} b_k^\alpha < \infty$. This gives various examples of quasi-infinitely divisible distributions with infinite negative part of the quasi-L\'evy measure and prescribed integrability conditions of the quasi-L\'evy measure around 0.
\end{ex}

So far we have identified various quasi-infinitely divisible distributions and given examples of distributions that are not quasi-infinitely divisible.  Cuppens \cite[Thm.~4.3.4]{Cuppens1975} shows that $(0,\nu,\gamma)_c$, where $\nu$ is a finite quasi-L\'evy type measure,  is the characteristic triplet of a quasi-infinitely divisible distribution if and only if $\exp (\nu) := \sum_{n=0}^\infty (1/n!) \nu^{\ast n}$ is a measure. However, it is in general difficult to check if the exponential of a finite signed measure is a measure.  In
\cite[Sect. 5]{Cuppens1969}, Cuppens raised the question of characterising all quasi-infinitely divisible distributions with Gaussian variance zero and finite quasi-L\'evy measure. We do not provide an answer to this question, but at least characterise in Theorem \ref{thm-integers3} (in combination with Theorem \ref{thm-integers}) all quasi-infinitely divisible distributions with zero Gaussian variance and quasi-L\'evy measure being concentrated on $\bZ$.

Finally, we mention that, using P\'olya's theorem employed in Example \ref{ex-linnik}, we can construct further (symmetric) quasi-infinitely divisible distributions:

\begin{ex} \label{ex-Polya}
Let $\nu_1 \cl \mcal B_0\to \R$ be a quasi-L\'evy type measure such that
$\int_{\bR}  (x^2\vee |x|)  \, |\nu_1|(\di x) < \infty$.  Suppose that $\nu_1$ is symmetric
(i.e.\ $\nu_1(B)=\nu_1(-B)$ for $\forall B\in\mcal B_0$). Let $\nu_2(\di x) = \pi^{-1} x^{-2}\, \di x$, the L\'evy measure
of the standard Cauchy distribution, and $c(x) = x \mathbf{1}_{[-1,1]}(x)$.
We claim that then $(a,\nu_1 + \lambda \nu_2, \gamma)_c$ is the characteristic triplet of some
quasi-infinitely divisible distribution whenever $a\geq 0$, $\gamma \in \bR$ and $\lambda > 0$ is sufficiently large.
To see this, it is obviously sufficient to consider the case $\gamma=0$. Let
$$h(z) = \int_\bR (\re^{\ri x z} - 1 - \ri x z \mathbf{1}_{[-1,1]}(x)) \, \nu_1(\di x), \quad z\in \bR.$$
By symmetry of $\nu_1$, $h$ is real-valued, even, continuous and $h(0) = 0$. Using dominated convergence and the integrability condition on $|\nu_1|$, $h$ is twice differentiable with derivatives
$$h'(z) = \int_\bR \ri x (\re^{\ri x z} - \mathbf{1}_{[-1,1]}(x)) \, \nu_1(\di x) \quad \mbox{and} \quad h''(z) = - \int_\bR x^2 \, \re^{\ri x z} \, \nu_1(\di x), \quad z\in \bR,$$
so that $h'$ and $h''$ are bounded. Further, $h(z) = O(z)$ as $z\to \infty$ by Lemma 43.11~(ii) in \cite{Sa},
applied to $\nu_1^+$ and $\nu_1^-$ separately. Let
$$
\varphi_\lambda(z) = \exp (-\lambda |z| + h(z)).
$$
An application of P\'olya's theorem in the form of \cite[Cor.\ 2 to Thm.\ 1.2.2]{Lukacs} shows that $\ph_{\ld} (z)$ is
the characteristic function of a probability distribution for sufficiently large $\lambda>0$. Hence
$$\exp \left( - a z^2/2 + \int_\bR (\re^{\ri x z} - 1- \ri x z \mathbf{1}_{[-1,1]}(x)) \, (\nu_1 + \lambda \nu_2)(\di x) \right) = \re^{-a z^2/2} \varphi_{\lambda} (z)$$
is the characteristic function of a probability distribution for large enough $\lambda$. This example shows in particular that for every symmetric and singular (with respect to Lebesgue measure) measure $\rho$ on $\bR$ with $\int_{\R} (x^2 \vee |x|) \, \rho(\di x) < \infty$,
there exists a quasi-infinitely divisible distribution with Gaussian variance 0 and quasi-L\'evy measure $\nu$ such that $\nu^- = \rho$.
\end{ex}


\section{Convergence of quasi-infinitely divisible distributions} \label{S4}

In this section we study weak convergence of a sequence of quasi-infinitely divisible distributions.
Recall that a sequence $(\mu_n)_{n\in \bN}$ of probability measures on $\bR$ converges weakly to a probability measure $\mu$, if
\begin{equation} \label{eq-weak-conv}
\lim_{n\to\infty} \int_\bR f(x) \, \mu_n(\di x) = \int_\bR f(x) \, \mu(\di x), \qquad \forall\, f\in C_b(\bR; \bR),
\end{equation}
where $C_b(\bR; \bR)$ denotes the class of real-valued bounded continuous functions on $\bR$. Recall that the class of infinitely divisible distributions is closed under weak convergence, see e.g.\ \cite[Lem.~7.8]{Sa}. In contrast, it is easy to see that the class of quasi-infinitely divisible distributions is \emph{not} closed under weak convergence. For example,  $b(1,p)$ is quasi-infinitely divisible if and only if $p\neq 1/2$ by Example \ref{ex-bin}, and by letting $p\to 1/2$ we can represent the non-quasi-infinitely divisible distribution $b(1,1/2)$ as a weak limit of quasi-infinitely divisible distributions. By applying Corollary \ref{cor-finite3} we can show even more, namely that the class of quasi-infinitely divisible distributions is dense in the class of distributions.

\begin{thm} \label{thm-dense}
The class of quasi-infinitely divisible distributions on $\bR$ with finite quasi-L\'evy measure and zero Gaussian variance
is dense in the class of probability distributions on $\bR$ with respect to weak convergence.
\end{thm}

\begin{proof}
Let $\mu$ be a probability distribution. For $n\in \bN$ let $b_{j,n} = -n + j/n$, $j\in \{0, \ldots, 2n^2\}$ and define the discrete distribution $\mu_n$, concentrated on the lattice\linebreak
$\{b_{0,n}, b_{1,n}, \ldots, b_{2n^2,n}\}$ by
$$\mu_n ( \{ b_{j,n} \})  = \begin{cases} \mu( (-\infty, b_{0,n}] ), & j = 0, \\
\mu (b_{j-1,n} , b_{j,n}]), & j=1, \ldots, 2n^2-1, \\
\mu ( (b_{2n^2-1,n}, \infty)), & j=2n^2.
\end{cases}$$
Then $$\mu_n ((-\infty, b_{j,n}]) = \mu ((-\infty, b_{j,n}]), \quad j\in \{0,\ldots, 2n^2-1\},$$
and from this it follows easily that $\mu_n((-\infty,x])$ converges to $\mu((-\infty,x])$ as $n\to\infty$ at every continuity point $x$ of the distribution function of $\mu$. Hence $\mu_n \stackrel{w}{\to} \mu$ as $n\to\infty$. It hence suffices to show that every distribution $\mu_n$ is a weak limit of quasi-infinitely divisible distributions with finite quasi-L\'evy measure and Gaussian variance 0. To see this, observe first that every distribution concentrated on $\{b_{0,n}, \ldots, b_{2n^2,n}\}$ can arbitrarily well be approximated by distributions $\sigma$ concentrated on $\{b_{0,n}, \ldots, b_{2n^2,n}\}$ such that $\sigma (\{b_{j,n}\}) > 0$ for all $j\in \{0,\ldots, 2n^2\}$. Hence, we may restrict attention to such distributions $\sigma$.
If the characteristic function of $\sigma$ has no zeroes, then $\sigma$ will be quasi-infinitely divisible with finite quasi-L\'evy measure by Corollary \ref{cor-finite3} and we are done. So suppose that $\widehat{\sigma}$ has zeroes.
Let $X$ be a random variable with distribution $\sigma$ and define $Y = nX+n^2$. Then $Y$ is concentrated on $\{0, 1, \ldots, 2n^2\}$ with
masses $a_j = P(Y=j) > 0$ for $j=0,\ldots, 2n^2$, and its characteristic function has zeroes.
Then the polynomial $f(w) = \sum_{j=0}^{2n^2} a_j w^j$ has zeroes on the unit circle. Factorising we can write $f(w) = a_{2n^2} \prod_{j=1}^{2n^2} (w- \xi_j)$. Now let
\begin{equation} \label{eq-modified}
f_h(w) = a_{2n^2} \prod_{j=1}^{2n^2} (w-\xi_j - h), \quad w\in \bC,
\end{equation}
for $h>0$. Then $f_h$ will not have zeroes on the unit circle for small enough $h$, and since the non-real zeroes of $f$ appear in pairs of complex conjugates, $f_h$ is a polynomial with real coefficients, say $f_h (w) = \sum_{j=0}^{2n^2} \alpha_{h,j} w^j$ with $\alpha_{h,j} \in \bR$. For small enough $h$, $\alpha_{h,j}$ will be close to $a_j$ which is strictly positive, hence also $\alpha_{h,j} > 0$. Now let $Z_h$ be a random variable with distribution $\sigma_h = \left(\sum_{j=0}^{2n^2} \alpha_{h,j} \right)^{-1} \sum_{j=0}^{2n^2}\alpha_{h,j} \delta_j$, and define $X_h = n^{-1} (Z_h -n^2)$. Then the characteristic function of $X_h$ has no zeroes for small enough $h$, and $X_h$ converges in distribution to $X$ as $h\downarrow 0$. Since $X_h$ is quasi-infinitely divisible with finite quasi-L\'evy measure and Gaussian variance 0 by Corollary \ref{cor-finite3}, the claim follows.
\end{proof}

Since the class of quasi-infinitely divisible distributions is not closed but dense, a handy characterisation of weak convergence of quasi-infinitely divisible distributions in terms of the characteristic triplet seems hard. Nevertheless, we aim at giving some easy sufficient and some necessary conditions in terms of the characteristic pair.
We say that a sequence $(\mu_n)_{n\in \bN}$ of  \emph{finite signed measures} on $\bR$ \emph{converges weakly} to a finite signed measure $\mu$ on $\bR$, if \eqref{eq-weak-conv} holds, and we denote this by $\mu_n \stackrel{w}{\to} \mu$; observe that also other (non-equivalent) definitions of weak convergence of signed measures can be found in the literature, see e.g. Section 2.6 in Cuppens \cite{Cuppens1975}, but we use this notion as for example done in Bogachev \cite[Def. 8.1.1]{Bogachev2}. The sequence $(\mu_n)_{n\in \bN}$ of finite signed measures is \emph{uniformly bounded}, if $(|\mu_n|)_{n\in \bN}$ is uniformly bounded, i.e. if
$$\sup_{n\in \bN} |\mu_n|(\bR) < \infty.$$
Finally, $(\mu_n)_{n\in \bN}$ is \emph{tight} if $(|\mu_n|)_{n\in \bN}$ is tight, i.e. if for every $\varepsilon > 0$ there exists some compact set $K\subset \bR$ such that
$$\sup_{n\in \bN} |\mu_n|(\bR \setminus K) \leq \varepsilon.$$
A weakly convergent sequence of finite signed measures must necessarily be uniformly bounded and  tight, see
Bogachev \cite[Thm.\ 8.6.2]{Bogachev2}.

Weak convergence of infinitely divisible distributions can be described by convergence properties of characteristic triplets as in \cite[Thm. 8.7]{Sa}, but in dimension 1 it is often easier to work with characteristic pairs. The following result, originally due to Gnedenko,
is found in Gnedenko and Kolmogorov \cite[Section 19, Thm.~1]{GK}.

\begin{thm} \label{thm-Gnedenko}
Let $c:\bR \to \bR$ be a fixed representation function that additionally is continuous, so that $g_c(\cdot , z)$ defined by \eqref{eq-def-g} is continuous for each fixed $z$. Let $(\mu_n)_{n\in  \bN}$ be a sequence of infinitely divisible distributions with characteristic pairs $(\zeta_n,\gamma_n)_c$. Then $(\mu_n)_{n\in \bN}$ converges weakly if and only if $(\zeta_n)_{n\in \bN}$ converges weakly to some finite measure $\zeta$ and $\gamma_n$ converges to some $\gamma\in \bR$. In that case, the weak limit $\mu$ is infinitely divisible and has
characteristic pair $(\zeta,\gamma)_c$.
\end{thm}

As already mentioned, a similarly neat characterisation of weak convergence of quasi-infinitely divisible distributions is not to be expected, but at least we have the following result:

\begin{thm} \label{thm-conv}
Let $c$ be a continuous representation function and let $(\mu_n)_{n\in \bN}$ be  a sequence of quasi-infinitely divisible distributions with characteristic pairs $(\zeta_n,\gamma_n)_c$.\\
(a) Suppose that $\gamma_n$ converges to some $\gamma\in \bR$ and that $\zeta_n$ converges weakly to some finite signed measure $\zeta$ as $n\to \infty$. Then $\mu_n$ converges weakly to a quasi-infinitely divisible distribution $\mu$ with characteristic pair $(\zeta,\gamma)_c$.\\
(b) Suppose that $\mu_n$ converges weakly to some distribution $\mu$ as $n\to \infty$ and that $(\zeta_n^-)_{n\in \bN}$ is tight and uniformly bounded. Then $\mu$ is quasi-infinitely divisible, and if $(\zeta,\gamma)_c$ denotes the characteristic pair of $\mu$, then $\gamma_n \to \gamma$ and $\zeta_n \stackrel{w}{\to} \zeta$ as $n\to\infty$.\\
(c) If $(\mu_n)_{n\in \bN}$ is tight and $(\zeta_n^-)_{n\in \bN}$ is tight and uniformly bounded, then $(\gamma_n)_{n\in \bN}$ is bounded and $(\zeta_n^+)_{n\in \bN}$ as well as $(|\zeta_n|)_{n\in \bN }$ are tight and uniformly bounded.\\
(d) If $(\gamma_n)_{n\in \bN}$ is bounded and $(\zeta_n)_{n\in \bN}$ is tight and uniformly bounded, then $(\mu_n)_{n\in \bN}$ is tight.
\end{thm}

\begin{proof}
(a) Suppose that $\zeta_n \stackrel{w}{\to} \zeta$ and $\gamma_n \to \gamma$ as $n\to\infty$. Observe that
$$
\widehat{\mu}_n(z)  =  \exp \left( \ri \gamma_{n} z + \int_{\bR} g_c(x,z) \, \zeta_n(\di x) \right) .
$$
Since $g_c(\cdot, z)$ is continuous and bounded, we have
$$\widehat{\mu}_n(z) \to \exp \left( i \gamma z + \int_{\bR} g_c(x,z) \, \zeta(d\di x) \right).$$
The right-hand side of this equation is continuous in $z$ and takes the value 1 at $z=0$. By L\'evy's continuity theorem, it is the characteristic function of some probability distribution $\mu$, and $\mu_n \stackrel{w}{\to} \mu$ as $n\to\infty$.
Then clearly $\mu \sim \qid (\zeta,\gamma_c)_c$.

(c) Let $(n')$ be an arbitrary subsequence of $(n)$. Since $(\mu_{n'})$ is tight and $(\zeta_{n'}^-)$ is tight and uniformly bounded, there exists a further subsequence $(n'')$ of $(n')$ such that $\mu_{n''}$ and $(\zeta_{n''}^-)$ converge weakly, cf. \cite[Thm.~8.6.2]{Bogachev2}.
Denote the limits by $\mu$ and $\xi$, respectively.
Let $\rho_{n''}$ be an infinitely divisible distribution with characteristic pair $(\zeta_{n''}^-, 0)_c$. By Theorem \ref{thm-Gnedenko}, $(\rho_{n''})$ converges weakly to some infinitely divisible distribution $\rho$ with characteristic pair $(\xi,0)_c$. Hence also $\mu_{n''} \ast \rho_{n''}$ converges weakly to $\mu \ast \rho$, and since $\mu_{n''} \ast \rho_{n''}$ is infinitely divisible with characteristic pair $(\zeta_{n''}^+, \gamma_{n''})_c$, it follows from Theorem \ref{thm-Gnedenko} that $\zeta_{n''}^+$ converges weakly and that $\gamma_{n''}$ converges.

We have shown that every subsequence $(n')$ of $(n)$ contains a further subsequence $(n'')$ such that $\zeta_{n''}^+$ converges weakly and such that $\gamma_{n''}$ converges. It follows that $(\gamma_n)_{n\in \bN}$ must be bounded, and that $(\zeta_n^+)_{n\in \bN}$ is tight and uniformly bounded, the latter by \cite[Thm.~8.6.2]{Bogachev2}.
It follows from \eqref{eq-Hahn} that also $(|\zeta_n|)_{n\in \bN}$ is then tight and uniformly bounded.

(b) Suppose that  $(\mu_n)_{n\in \bN}$ converges weakly to $\mu$ and that $(\zeta_n^-)_{n\in \bN}$ is tight and uniformly bounded. Then $(\mu_n)_{n\in \bN}$ is also tight, and it follows from the already proved part (c) that $(\gamma_n)_{n\in \bN}$ is bounded and that $(\zeta_n^+)_{n\in \bN}$ as well as $(|\zeta_n|)_{n\in \bN}$ are tight and uniformly bounded. We claim that $(\gamma_n)_{n\in \bN}$ converges to some constant $\gamma$ and that $(\zeta_n)_{n \in \bN}$ converges weakly to some finite signed measure $\zeta$. For if this was not the case, then by tightness and (uniform) boundedness we could find two subsequences $(\zeta_{n_{k,1}}, \gamma_{n_{k,1}})_{k\in \bN}$ and $(\zeta_{n_{k,2}}, \gamma_{n_{k,2}})_{k\in \bN}$ such that $\zeta_{n_{k,1}} \stackrel{w}{\to} \zeta^1$, $\zeta_{n_{k,2}} \stackrel{w}{\to} \zeta^2$, $\gamma_{n_{k,1}} \to \gamma^1$ and $\gamma_{n_{k,2}} \to \gamma^2$ as $k\to\infty$, but such that $\zeta^1 \neq \zeta^2$ or $\gamma^1 \neq \gamma^2$. It then follows from part (a) that $\mu_{n_{k,1}}$ and $\mu_{n_{k,2}}$ converge to $\qid (\zeta^1,\gamma^1)_c$ and $\qid (\zeta^2,\gamma^2)_c$, respectively, which must be different by the uniqueness of the characteristic pair. This contradicts that $(\mu_n)_{n\in \bN}$ is weakly convergent, and it follows that $\zeta_n \stackrel{w}{\to} \zeta$ and $\gamma_n \to \gamma$ as $n\to\infty$ for some finite signed measure $\zeta$ and some $\gamma \in \bR$. Hence $\mu$ is quasi-infinitely divisible with characteristic pair $(\zeta,\gamma)_c$ by part (a).

(d) Let $(n')$ be a subsequence of $(n)$. By tightness and (uniform) boundedness, there exists a subsequence $(n'')$ such that $\zeta_{n''}$ converges weakly to some finite signed measure $\zeta$ (cf.\ \cite[Thm.~8.6.2]{Bogachev2}) and $\gamma_{n''}$ converges to some $\gamma\in \bR$. By part (a), this shows that $\mu_{n''}$ converges weakly. Hence every subsequence of $(\mu_n)$ has a weakly convergent subsequence, so that $(\mu_n)_{n\in \bN}$ is tight (e.g. \cite[Thm.~8.6.2]{Bogachev2}).
\end{proof}

We have already seen that the sequence of quasi-infinitely divisible Bernoulli distributions $b(1,1/2 + 1/n)$ converges weakly to the non-quasi-infinitely divisible Bernoulli distribution $b(1,1/2)$ as $n\to\infty$; from Theorem \ref{thm-cuppens} we also see that $b(1,1/2+1/n)$ has the quasi-L\'evy measure $\sum_{m=1}^\infty m^{-1} (-1)^{m+1} \left( \frac{n-2}{n+2} \right)^m \delta_{-m}$. The signed measure $\zeta_n$ in the characteristic pair of $b(1,1/2+1/n)$ coincides with the quasi-L\'evy measure, and it is easy to see that $(\zeta_n^-)_{n\in \bN}$ and hence $(|\zeta_n|)_{n\in \bN}$ are neither uniformly bounded nor tight. As the limit is not quasi-infinitely divisible, this is not surprising. It is natural to ask if convergence of $\mu_n$ to a quasi-infinitely divisible distribution implies uniform boundedness or tightness of the signed measures in the characteristic pair. That this is not the case, even if the limit is infinitely divisible, is shown in the next example.

\begin{ex} \label{ex-non-conv}
Let $\sigma(\di x) = (1/2)\, \re^{-|x|}\, \di x$, a symmetric two-sided exponential distribution,
and let $\mu = (1/2) \,\delta_0 + (1/2) \,\sigma$. It is known that $\sigma$ is infinitely divisible with
$$\widehat{\sigma}(z) = \frac{1}{1+z^2} = \exp\left( \int_{-\infty}^\infty \left( \re^{\ri x z} - 1\right) |x|^{-1} \re^{-|x|} \,
\di x\right), \quad z\in \bR,$$
cf. Steutel and van Harn \cite[Ex. IV.29, IV.4.8]{SvH} or \cite[Ex.\ 15.14]{Sa}. Hence
\begin{eqnarray*}
\widehat{\mu}(z) & = & \frac12 \left( 1 + \frac{1}{1+z^2} \right) = \frac{ 1 + \left( {z}/{\sqrt{2}} \right)^2}{1+z^2} = \frac{\widehat{\sigma}(z)}{\widehat{\sigma}(z/\sqrt{2})} \\
& = & \exp \left( \int_{-\infty}^\infty \left( \re^{\ri x z} - 1 \right) \, \frac{\re^{-|x|}}{|x|} \, \di x - \int_{-\infty}^\infty \left( \re^{\ri y z} - 1 \right) \frac{\re^{-\sqrt{2} |y|}}{\sqrt{2} |y|} \, \sqrt{2} \, \di y \right) \\
& = & \exp \left( \int_{-\infty}^\infty \left( \re^{\ri x z} - 1 \right) \frac{\re^{-|x|} - \re^{-\sqrt{2}|x|}}{|x|} \, \di x\right) , \quad z\in \bR,
\end{eqnarray*}
showing that $\mu$ is infinitely divisible with finite L\'evy measure $|x|^{-1} (\re^{-|x|} - \re^{-\sqrt{2}|x|})\, \di x$, drift 0  and Gaussian variance 0. We will now approximate $\mu$ by a sequence of quasi-infinitely divisible distributions whose signed  measures in the characteristic pairs are neither tight nor uniformly bounded. To do so, we choose for each $n\in \bN$ a finite sequence $b_{n,1} < b_{n,2} < \ldots < b_{n,m(n)}$ such that
$$|b_{n,1} - (-n)| < 1/n, \; |b_{n,m(n)} - n| < 1/n, \; | b_{n,j+1} - b_{n,j}| < 1/n , \; \forall j\in \{1,\ldots, m(n)-1\},$$
and such that $\{ b_{n,1}, \ldots, b_{n,m(n)} \}$ is linearly independent over $\bQ$, i.e.\ such that\linebreak
$\sum_{j=1}^{m(n)} l_j b_{n,j} = 0$ with $l_1, \ldots, l_{m(n)} \in \bQ$ implies $l_1 = \ldots = l_{m(n)} = 0$; this is obviously possible, since every nontrivial subinterval of $\bR$ is uncountable. Now define
\begin{eqnarray*}
a_{n,1} & := & \sigma ((-\infty, b_{n,1}]), \quad a_{n,m(n)} := \sigma ( (b_{n,m(n)-1}, \infty)) ,\\
a_{n,j} & := & \sigma ((b_{n,j-1}, b_{n,j}]) \quad \mbox{for} \quad j\in \{2,\ldots, n(m) - 1\}
\end{eqnarray*}
and
$$\sigma_n := \sum_{j=1}^{m(n)} a_{n,j} \delta_{b_{n,j}}, \quad \mu_n := \left( \frac12 + \frac{1}{n} \right) \delta_0 + \left( \frac12 - \frac{1}{n} \right) \sigma_n, \quad n\geq 3.$$
Then $\sigma_n \stackrel{w}{\to} \sigma$ and hence $\mu_n \stackrel{w}{\to} \mu$ as $n\to\infty$. Observe that by Theorem \ref{thm-cuppens}, $\mu_n$ is quasi-infinitely divisible with Gaussian variance 0, drift 0 and finite quasi-L\'evy measure $\nu_n$ given by
$$\nu_n := \sum_{j=1}^\infty j^{-1} (-1)^{j+1} \left( \frac{n-2}{n+2} \right)^j \sigma_n^{\ast j}.$$ Next, observe that $\sigma_n$ is concentrated on $ \Lambda_{n,1} := \{b_{n,1}, \ldots, b_{n,m(n)}\}$, hence $\sigma_n^{\ast j}$ is concentrated on $\Lambda_{n,j} := \{b_{n,r_1} + b_{n,r_2} + \ldots + b_{n,r_j} : r_1,\ldots, r_j \in \{1,\ldots, m(n)\}$. From the linear independence over $\bQ$ of $\Lambda_{n,1}$ it then follows that $\Lambda_{n,j}$ and $\Lambda_{n,j'}$ are disjoint for $j\neq j'$. Hence
$$\nu_n^+ = \sum_{j=1}^\infty \frac{1}{2j-1} \left( \frac{n-2}{n+2} \right)^{2j-1} \sigma_n^{\ast (2j-1)} \quad \mbox{and} \quad \nu_n^- = \sum_{j=1}^\infty \frac{1}{2j} \left( \frac{n-2}{n+2} \right)^{2j} \sigma_n^{\ast (2j)}.$$
Let $K\in \bN$.
To show that $\lim_{n\to\infty} \nu_n^- (\bR \setminus [-K,K]) = +\infty$, let $X_{n,1}, \ldots, X_{n,j}, Y_{n,1}, \ldots, Y_{n,j}$ be independent and identically distributed random variables with distribution $\sigma_n$. Since $P(X_{n,1} \leq 1/2) \geq 1/2$ and $P(X_{n,1} \geq -1/2) \geq 1/2$, it follows from the symmetrization inequalities in \cite[Lemmas V.5.1, V.5.2]{Feller2} that for every $j\in \bN$ and $n\geq 2K+1$ we have
\begin{eqnarray*}
\sigma_n^{\ast j} (\bR \setminus [-K,K]) & = & P( |X_{n,1}+ \ldots + X_{n,j}| > K) \\
& \geq  & \frac12 P( |(X_{n,1} - Y_{n,1}) + \ldots + (X_{n,j} - Y_{n,j})| > 2K) \\
& \geq & \frac14 P( |X_{n,1} - Y_{n,1}| > 2K) \\
& \geq & \frac18 P \left(|X_{n,1}| > 2K+\frac12 \right) \geq \frac18 \int_{2K+1}^\infty \re^{-x}\, \di x.
\end{eqnarray*}
Hence
$$\nu_n^- (\bR \setminus [-K,K]) \geq \frac18 \int_{2K+1}^\infty \re^{-x} \, \di x \sum_{j=1}^\infty \frac{1}{2j} \left( \frac{n-2}{n+2} \right)^{2j} \to +\infty \quad \mbox{as $n\to\infty$.}$$
Defining $\zeta_n := (1\wedge x^2) \, \nu_n(\di x)$, it follows that $\zeta_n^- (\bR \setminus [-K,K]) \to \infty$ as $n\to\infty$. In particular, $(\zeta_n^-)_{n\in \bN}$ is neither uniformly  bounded nor tight, hence also $(\zeta_n)_{n\in \bN}$ is neither uniformly bounded nor tight. This also shows that $\zeta_n$ does not converge weakly, since every weakly convergent sequence of finite signed measures must be uniformly bounded (cf. \cite[Thm.~8.6.2]{Bogachev2}). In particular, $\zeta_n$ does not weakly converge to $(1\wedge x^2) |x|^{-1} (\re^{-|x|} - \re^{-\sqrt{2}|x|}) \di x$, although $\mu_n \stackrel{w}{\to} \mu$ and $\mu$ is infinitely divisible.
\end{ex}

When restricting attention to quasi-infinitely divisible distributions concentrated on the integers $\bZ$, phenomena like in Example \ref{ex-non-conv} do not occur and a complete characterisation of weak convergence in terms of the characteristic pair is possible. This will be treated in Theorem \ref{thm-integers-convergence}.


\section{Support properties of quasi-infinitely divisible distributions} \label{S5}

A striking difference between infinitely divisible distributions and quasi-infinitely divisible distributions is that a non-degenerate infinitely divisible distribution must necessarily have unbounded support (cf. \cite[Cor. 24.4]{Sa}), while there are many non-degenerate quasi-infinitely divisible distributions with bounded support as can be seen from Theorem \ref{thm-finite}.

For infinitely divisible distributions, many properties of the support can be described in terms of the characteristic triplet. For instance, an infinitely divisible distribution $\mu$ with characteristic triplet $(a,\nu,\gamma)_c$ is bounded from below if and only if $a=0$, $\supp \nu \subset [0,\infty)$ and $\int_0^1 x\, \nu(\di x) < \infty$ (cf. \cite[Thm. 24.7]{Sa}). Such a  characterisation cannot hold for quasi-infinitely divisible distributions, as can be seen e.g. by considering the binomial distribution $b(1,p)$ with $p\neq 1/2$, which is quasi-infinitely divisible, concentrated on $\{0,1\}$ and hence has bounded support. On the other hand, when $p\in (0,1/2)$, then the quasi-L\'evy measure $\nu$ is concentrated on $\bN$, and when $p\in (1/2,1)$, then $\nu$ is concentrated on $-\bN$, as follows from Theorem \ref{thm-cuppens}.  However, we can give at least the following result regarding the interplay between the supports of $\mu$, $\nu^-$ and $\nu^+$. Recall the definition of the Laplace transform $\bL_\mu(u) = \int_\bR \re^{-ux} \, \mu(\di x)$ for $u\geq 0$.

\begin{prop} \label{prop-support1}
Let $\mu$ be a quasi-infinitely divisible distribution with characteristic triplet $(a,\nu,\gamma)_c$. Then the following are equivalent:
\begin{itemize}
\item[(i)] $\mu$ is bounded from below, $\supp \, \nu^- \subset [0,\infty)$ and $\int_{(0,1)} x\, \nu^-(\di x) < \infty$.
\item[(ii)] $a=0$, $\supp \, \nu^+ \subset [0,\infty)$ and $\int_{(0,1)} x\, \nu^+(\di x) < \infty$.
\end{itemize}
If one (hence both) of the above conditions are satisfied, denote by $\gamma_0$ the drift of $\mu$. Then the Laplace transform $\bL_\mu$ of $\mu$ is given by
\begin{equation} \label{eq-Laplace1}
\bL_\mu(u) =  \exp \left( - \gamma_0 u - \int_0^\infty (1-\re^{-ux}) \, \nu (\di x) \right), \quad u\geq 0,
\end{equation}
and we have
$$\gamma_0 = \inf (\supp  \mu).$$
\end{prop}

\begin{proof}
Let $X,Y,Z$ be random variables with $\law(X) = \mu$, $\law(Y) \sim \qid (0,\nu^-, 0)_c$, $\law(Z) \sim \qid (a,\nu^+,\gamma)_c$ and such that $X$ and $Y$ are independent. Then \eqref{eq-factor1} holds. From the above mentioned characterisation of the support of infinitely divisible distributions we then have
\begin{equation*}
\mbox{\rm (i)}  \Longleftrightarrow   \mbox{$X$ and $Y$ bounded from below}
\Longleftrightarrow  \mbox{$Z$ bounded from below}
 \Longleftrightarrow  \mbox{\rm (ii)}.
\end{equation*}
If (i) and (ii) are satisfied, then $\mu$ has drift and $\mu \sim \qid (0,\nu,\gamma_0)_0$. Choosing $Y$ and $Z$ as above with respect to $c(x) =0$, i.e. $\law(Y) \sim \qid (0,\nu^-,0)_0$ and $\law(Z) \sim \qid (0,\nu^+,\gamma_0)_0$, the Laplace transforms of $Y$ and $Z$ are given by
$\bE \re^{-uY} = \exp \left( -\int_0^\infty (1-\re^{-ux} ) \, \nu^-(\di x) \right)$ {and} $\bE \re^{-uZ} = \exp \left( -\gamma_0 u - \int_0^\infty (1-\re^{-ux} ) \, \nu^+(\di x) \right)$,
respectively (e.g.\ \cite[Th.~
24.11]{Sa}). This gives \eqref{eq-Laplace1} since $\bE \re^{-u X} \, \bE \re^{-uY} =  \bE \re^{-uZ}$. Finally, we have\; $\inf \supp \law(Y) = 0$ and\; $\inf \supp \law(X) = \gamma_0$ by \cite[Cor. 24.8]{Sa}, so that\; $\inf (\supp \mu) = \gamma_0$ by \cite[Lem. 24.1]{Sa}.
\end{proof}

Infinite divisibility of a  distribution concentrated on $[0,\infty)$ can be characterized by the form of the Laplace transform (e.g. \cite[before Thm. 51.1]{Sa}). Under extra conditions, a characterisation in this vein can also be obtained for quasi-infinitely divisible distributions:

\begin{prop} \label{prop-support2}
Let $\gamma_0 \in \bR$ and $\nu\cl \mcal B_0 \to \R$ be a quasi-L\'evy type measure
with $\supp \nu \subset [0,\infty)$ and $\int_{(0,1)} x |\nu|(\di x) < \infty$. Let $\mu$ be a distribution on $\bR$. Then the following are equivalent:
\begin{enumerate}
\item[(i)] $\mu$ is bounded from below and quasi-infinitely divisible with characteristic triplet $(0,\nu,\gamma_0)_0$.
\item[(ii)] The Laplace transform of $\mu$ is finite for $u\geq 0$ and has the representation \eqref{eq-Laplace1}.
\end{enumerate}
\end{prop}

\begin{proof}
That (i) implies (ii) follows from Proposition \ref{prop-support1}. To prove the converse, suppose that $\bL_\mu(u)= \int_\bR \re^{-u x} \, \mu(\di x) < \infty$ for $u\geq 0$. Then $g$, defined by
$$g(u + \ri v) = \int_\bR \re^{- (u+\ri v) x} \, \mu(\di x)$$
exists in $\bC$ for $u\geq 0$ and $v\in \bR$, we have $g(u) = \bL_\mu(u)$ for $u\geq 0$  and by standard theorems on parameter dependent integrals
(e.g.\ \cite[IV \S 5 Section 4]{Elstrodt}), $g$ is continuous on $\{ w \in \bC\cl \Re (w) \geq 0\}$ and holomorphic on $\{ w  \in \bC \cl \Re (w) > 0\}$. Similarly, since $\int_0^\infty (1\wedge x) \, |\nu|(\di x)< \infty$,
$$f (u + \ri v) = \exp \left( - \gamma_0 (u + \ri v) - \int_0^\infty (1 - \re^{-(u+\ri v)x}) \, \nu(\di x) \right), \quad u, v \in \bR, u \geq 0,$$
defines a continuous function on $\{ w \in \bC\cl \Re (w) \geq 0\}$ that is holomorphic on $\{ w \in \bC\cl \Re(w) > 0\}$.
Since $f$ and $g$ agree on $\{w \in \bC\cl \Re (w) \geq 0, \Im (w) = 0\}$, they agree on $\{ w \in \bC\cl \Re (w) > 0\}$
(e.g.\ \cite[Cor.~IV.3.8]{Conway1}) and by continuity then also on the imaginary axis. Hence  $\widehat{\mu}(v) = g(-\ri v) = f(-\ri v) =
\exp \left(  \ri \gamma_0 v + \int_0^\infty (\re^{\ri v x} - 1) \, \nu(\di x)\right)$ for $v\in \bR$, showing that
$\mu$ is quasi-infinitely divisible with characteristic triplet $(0,\nu,\gamma_0)_0$. By Proposition \ref{prop-support1},
$\mu$ is then also bounded from below.
\end{proof}

Quasi-infinitely divisible distributions supported on $[0,\infty)$ with some additional properties can be characterised in a similar way as infinitely divisible distributions supported on $[0,\infty)$; the following theorem hence is an analogue of Theorem 51.1 in \cite{Sa} for infinitely divisible distributions.

\begin{thm} \label{thm-support3}
Let $\mu$ be a distribution with $\supp \mu \subset [0,\infty)$.
Then the following are equivalent:
\begin{itemize}
\item[(i)] $\mu$ is quasi-infinitely divisible with
$\supp \nu^- \subset [0,\infty)$ and $\int_{(0,1)} x \, \nu^-(\di x) <
\infty$, where $\nu$  denotes the quasi-L\'evy measure of $\mu$.
\item[(ii)] $\mu$ is quasi-infinitely divisible with $a =  0$,
$\supp \nu^+ \subset [0,\infty)$ and $\int_{(0,1)} x \, \nu^+ (\di x) <
\infty$, where $\nu$ denotes the quasi-L\'evy measure of $\mu$ and $a$ its Gaussian variance.
\item[(iii)] There exists a constant $\gamma_0 \geq 0$ and a quasi-L\'evy type measure $\sigma$ with $\supp \sigma \subset [0,\infty)$ and $\int_{(0,1)} x  \, |\sigma|(\di x) <
\infty$ such that
\begin{equation} \label{eq-support2}
\int_{[0,x]} y\, \mu(\di y) = \int_{(0,x]} \mu( [0,x-y]) y \,
\sigma(\di y) + \gamma_0 \, \mu([0,x]), \quad \forall\; x > 0.
\end{equation}
\end{itemize} If some and hence all of the above equivalent conditions are satisfied, then $\sigma=\nu$ and
$\gamma_0$ is the drift of $\mu$.
\end{thm}

\begin{proof}
The equivalence of (i) and (ii) is Proposition \ref{prop-support1}, and that (i) and (ii) imply (iii) with $\sigma=\nu$ and $\gamma_0$ the drift follows in complete analogy to the corresponding proof for infinitely divisible distributions as given in \cite[Thm.~51.1]{Sa}, by observing that the convolution theorem also holds for finite signed measures (e.g. Cuppens \cite[Thm. 2.5.4]{Cuppens1975} and a similar reasoning as in the proof of Prop.~\ref{prop-support2} to switch from Laplace transforms of finite signed measures to their Fourier transforms).

To show that (iii) implies (i), denote $\widetilde{\sigma}(\di y) := \gamma_0 \delta_0 (\di y) + y \sigma(\di y)$. Then
\begin{equation} \label{eq-fu}
\int_\bR f (y) y\, \mu (\di y) = \int_\bR \int_\bR f(y+z) \, \widetilde{\sigma}(\di y) \, \mu(\di z)
\end{equation}
for all functions $f$ of the form $f=\alpha_0 \one_{[0,t_1]} + \sum_{i=1}^n \alpha_i \one_{(t_i, t_{i+1}]}$ with $\alpha_i \in \bR$ and $0< t_1 <\ldots < t_{n+1}$; for $n=0$ this follows from \eqref{eq-support2}, and for $n>0$ by linearity. Since for each $u>0$ the function $f_u$  defined by $f_u(x)= \re^{-u x} \one_{[0,\infty)}(x)$ can be represented as an increasing limit of functions of the form $\alpha_0 \one_{[0,t_1]} + \sum_{i=1}^n \alpha_i \one_{(t_i, t_{i+1}]}$, and since both $\int_\bR f_u(y) |y| \, \mu(\di y)$ and $\int_\bR \int_\bR f_u(y+z) |\widehat{\sigma}|(\di y) \, \mu(\di z)$ are finite,
Equation \eqref{eq-fu} also holds for $f_u$ by dominated convergence.
Considering $\mathbb{L}_\mu(u) = \int_{[0,\infty)} \re^{-u x} \, \mu(\di x)$, $u\geq 0$, \eqref{eq-fu} for $u>0$ gives
$$- \frac{\di}{\di u} \mathbb{L}_\mu(u) = \mathbb{L}_\mu(u) \int_{[0,\infty)} \re^{-yu} \, \widetilde{\sigma}(\di y),$$ hence
$$\frac{\di}{\di u} \log \mathbb{L}_\mu(u) = -    \int_{[0,\infty)} \re^{-yu} \, \widetilde{\sigma}(\di y) = - \gamma_0 - \int_{(0,\infty)} y \, \re^{-u y} \,  \sigma(\di y).$$ Since $\log \mathbb{L}_\mu$ is continuous on $[0,\infty)$ with $\log \mathbb{L}_\mu(0) = 0$, we obtain
$$\log \mathbb{L}_\mu(u) = - \gamma_0 u - \int_0^u \int_0^\infty y e^{-ty} \, \sigma(\di y) \, \di t = -\gamma_0 u - \int_0^\infty (1 - \re^{-u y}) \, \sigma(\di y), \quad u \geq 0,$$
showing that $\mu$ is quasi-infinitely divisible with characteristic triplet $(0,\sigma,\gamma_0)_0$ by Proposition \ref{prop-support2}.
\end{proof}

A characterisation in terms of the characteristic triplet for a   quasi-infinitely divisible distribution to be concentrated on the integers will be given in Theorem \ref{thm-integers3} below.


\section{Moments} \label{S6}

Recall that a function $h\cl \bR \to \bR$ is \emph{submultiplicative} if it is nonnegative and there is a constant $B>0$ such that
\begin{equation} \label{eq-subm1}
h(x+y) \leq B h(x) \, h(y), \quad \forall\; x, y \in \bR.
\end{equation}
Examples of submultiplicative functions can be found in \cite[Prop. 25.4]{Sa}, we only note that $x\mapsto (|x|\vee 1)^\alpha$ for $\alpha >0$, $x\mapsto \exp (\alpha |x|^\beta)$ for $\alpha > 0$ and $\beta \in (0,1]$, $x\mapsto \re^{\alpha x}$ for $x\in \bR$ and $x\mapsto \log (|x|\vee \re)$ are submultiplicative functions. We expect the following lemma to be well-known, but we were unable to find a ready reference and hence give a proof:

\begin{lem} \label{lem-subm1}
Let $h\cl \bR \to [0,\infty)$ be submultiplicative and $X$ and $Y$ be two real valued independent random variables. Then $\bE h(X+Y)$ is finite if and only if both $\bE h(X)$ and $\bE h(Y)$ are finite.
\end{lem}

\begin{proof}
If $\bE h(X) < \infty$ and $\bE h(Y) < \infty$, then
$\bE h(X+Y) \leq B \bE h(X) \bE h(Y)<\infty$ by \eqref{eq-subm1} and independence. Conversely, suppose that $\bE h(X+Y) < \infty$. If $h$ is equal to the zero-function, we have nothing to prove, so suppose that there is $x_0 \in \bR$ with $h(x_0) > 0$. From \eqref{eq-subm1} we then conclude $B h(x) h(x_0-x) \geq h(x_0) > 0$ so that $h(x) > 0$ for all $x\in \bR$. Further, for $x,y\in \bR$ we have
$h(x) = h(x+y - y) \leq B h(x+y) h(-y)$ so that $h(x) / h(-y) \leq B h(x+y)$. Hence $\bE h(X) \bE (1/h(-Y)) \leq B \bE h(X+Y) < \infty$ so that $\bE h(X) < \infty$ and similarly $\bE h(Y) < \infty$.
\end{proof}

For infinitely divisible distributions and submultiplicative functions, finiteness of $h$-moments can be characterised by the corresponding property of the L\'evy measure restricted to $\{x\in \bR\cl |x|> 1\}$ (cf.\ \cite[Thm.\ 25.3]{Sa}).  This is not true in complete generality for quasi-infinitely divisible distributions and arbitrary submultiplicative functions, as will be shown for exponential moments in Example \ref{ex-non-moments}, but at least one direction holds and we have the following result:

\begin{thm} \label{thm-moments1}
Let $\mu$ be a quasi-infinitely divisible distribution on $\bR$ with characteristic triplet $(a,\nu,\gamma)_c$ with respect to the
representation function $c(x)=x \bf{1}_{\{|x|\leq 1\}}$.\\
(a) Let  $h\cl \bR \to [0,\infty)$ be a submultiplicative function. Then the following are equivalent:
\begin{enumerate}
\item[(i)] $\mu$ and $(\nu^-)_{|\{x\in \bR\cl |x|> 1\}}$ have finite $h$-moments, i.e.\ $\int_\bR h(x) \, \mu(\di x) < \infty$ and $\int_{|x| >1} h(x) \, \nu^-(\di x) < \infty$.
\item[(ii)] $(\nu^+)_{|\{ x\in \bR\cl |x|> 1\}}$ has finite $h$-moment, i.e.\ $\int_{|x|>1} h(x) \, \nu^+(\di x) < \infty$.
\end{enumerate}
In particular, finiteness of the $h$-moment of $(\nu^+)_{|\{ x\in \bR\cl |x|> 1\}}$ implies finiteness of the $h$-moment of
$(\nu^-)_{|\{ x\in \bR\cl |x|> 1\}}$.\\
(b) Let $X$ be a random variable with distribution $\mu$ and let $\alpha \in \bR$. We then have
\begin{eqnarray*}
\bE (X) & = & \gamma + \int_{|x|>1} x \, \nu(\di x) = \gamma_m \quad \mbox{provided} \quad \int_{|x|>1} |x| \, \nu^+(\di x) < \infty, \\
\mbox{\rm Var} (X) & = & a + \int_\bR x^2 \, \nu(\di x) \quad \mbox{provided} \quad \int_{|x|>1} x^2 \, \nu^+(\di x) < \infty, \quad \mbox{and}\\
\bE (\re^{\alpha X}) & = & \exp \left( \alpha^2 a/2 + \int_\bR (\re^{\alpha x}  - 1 - \alpha x \mathbf{1}_{\{|x|\leq 1\}}) \, \nu (\di x) + \alpha \gamma\right)\\
& & \quad \quad \quad \quad \quad \quad \quad \mbox{provided} \quad \int_{|x|>1} \re^{\alpha x} \, \nu^+(\di x)   < \infty.
\end{eqnarray*}
Observe that $\gamma_m$ is the center of $\mu$ as defined in Remark \ref{rem-drift}.
\end{thm}

\begin{proof}
As before, let $X,Y,Z$ be random variables with $\law(X) = \mu$, $\law(Y) \sim \qid (0,\nu^-, 0)_c$, $\law(Z)\sim \qid (a,\nu^+,\gamma)_c$   and such that $X$ and $Y$ are independent. Then \eqref{eq-factor1} holds, i.e. $X+Y \stackrel{d}{=} Z$.

To prove (a), recall that an infinitely divisible distribution has finite $h$-moment if and only if the L\'evy measure restricted to $\{ x\in \bR\cl |x|> 1\}$ has finite $h$-moment (e.g. \cite[Thm. 25.3]{Sa}). Hence
$$\mbox{(i)} \Longleftrightarrow \bE h(X) < \infty \; \mbox{and} \; \bE h(Y) < \infty \Longleftrightarrow \bE h(X+Y) < \infty \Longleftrightarrow \mbox{(ii)},$$
where the equivalence in the middle follows from Lemma \ref{lem-subm1}.

The proof of (b) follows from (a), the fact that $\bE X + \bE Y  =  \bE Z$, $\mbox{\rm Var} (X) + \mbox{\rm Var} (Y)  =  \mbox{\rm Var} (Z)$,
$\bE \re^{\alpha X} \, \bE \re^{\alpha Y} =  \bE \re^{\alpha Z}$,
and the corresponding formulas for expectation, variance and exponential moments of the infinitely divisible distributions $Y$ and $Z$ given in \cite[Ex. 25.12 and Thm. 25.17]{Sa}.
\end{proof}

Finiteness of an exponential moment of a quasi-infinitely divisible distribution does not imply finiteness of the corresponding exponential moment of the
total variation of the restricted quasi-L\'evy measure. This is shown in the following example.

\begin{ex} \label{ex-non-moments}
Let $(b_n)_{n\in \bN}$ be a sequence of real numbers that is linearly independent over $\bQ$ and satisfies $b_n \in (2+n-1/4,2+n)$ for each $n\in \bN$; such a sequence obviously exists, since every non-degenerate interval is uncountable. Define the probability distribution
$$\sigma = \frac{11}{12} \delta_{b_1} + \sum_{n=2}^\infty 4^{-n} \delta_{b_n}.$$
Let $\lambda = \int_\bR \re^{x} \, \sigma (\di x)$.
Then (since $\re^{b_1} \geq \re^{2.75} >12$),
$$1 < \lambda = \frac{11}{12} \re^{b_1} + \sum_{n=2}^\infty \re^{b_n} 4^{-n} < \infty.$$
Let $p\in (1/2,1)$ such that $(1-p)/p \geq 1/\lambda$, which is possible since $\lambda > 1$. Define the probability distribution $\mu$ by
$$\mu = p \delta_0 + (1-p) \sigma.$$
By Theorem \ref{thm-cuppens}, $\mu$ is quasi-infinitely divisible with finite quasi-L\'evy measure $\nu =
\sum_{m=1}^\infty m^{-1} (-1)^{m+1} \left( (1-p)/p \right)^m \sigma^{\ast m}$. Since $\sigma$ has finite exponential moment $\int_\bR \re^{x} \, \sigma(\di x)$, so has $\mu$. However, $\int_{x>1} \re^{x} \nu^+(\di x) = \infty$ as we will now show: as in the proof of Example \ref{ex-non-conv}, by the linear independence over $\bQ$ of $(b_n)_{n\in \bN}$, the supports of $\sigma^{\ast m}$ are disjoint for different $m \in \bN$, hence
$$\nu^+ = \sum_{m\in \bN, m\, \rm{odd}} m^{-1} \left(\frac{1-p}{p}\right)^m \sigma^{\ast m}.$$
Since $\int_\bR \re^{x} \sigma^{\ast m} (\di x) = \left( \int_\bR \re^{x} \sigma(\di x) \right)^m = \lambda^m$, and since $\supp \sigma^{\ast m} \subset (1,\infty)$, this gives
$$\int_{\{x>1\}} \re^{x} \, \nu^+(\di x) =  \sum_{m\in \bN, m \,\rm{odd}} m^{-1} \left(\frac{1-p}{p}\right)^m \lambda^m = \infty$$
since $\lambda (1-p)/p \geq 1$. Hence $\int_{\{x>1\}} \re^{x} \nu^+ (\di x) = \infty$ (and similarly $\int_{\{x>1\}} \re^{x} \, \nu^-(\di x) = \infty$) although $\int_\bR \re^{x} \, \mu(\di x) < \infty$ and the function $x\mapsto \re^{x}$ is submultiplicative.
\end{ex}

For a quasi-infinitely divisible distribution concentrated on the integers it will be shown in Theorem \ref{thm-moments-integers} that finiteness of its $h$-moment can be characterised by finiteness of the $h$-moment of the total variation of its quasi-L\'evy measure, provided the function $h$ satisfies an additional condition, the GRS-condition defined in \eqref{eq-GRS} below. Observe that exponential functions do not satisfy the GRS-condition. If a characterisation as in Theorem \ref{thm-moments-integers} below holds for general quasi-infinitely divisible distributions when $h$ satisfies the GRS-condition, we do not know.


\section{Continuity properties} \label{S7}

In this section we shall give some sufficient conditions in terms of the characteristic triplet for a quasi-infinitely divisible distribution to have a Lebesgue density or to be continuous. The following result ensures densities and is in line with the corresponding results for infinitely divisible distributions by Orey, cf. \cite[Prop. 28.3]{Sa}.

\begin{thm} \label{thm-density}
Let $\mu$ be a quasi-infinitely divisible distribution with characteristic triplet $(a,\nu,\gamma)_c$ with respect to some $c$.
Suppose further that $a>0$ or
\begin{align} \label{eq-density1}
&\liminf_{r\downarrow 0} r^{-\beta} \int_{[-r,r]} x^2 \, \nu^+ (\di x)\\
\nonumber &\qq\qq > \limsup_{r\downarrow 0} r^{-\beta} \int_{[-r,r]}  x^2 \,
\nu^-(\di x) = 0 \quad \mbox{for some $\beta \in (0,2)$}.
\end{align}
Then $\mu$ has an infinitely often differentiable density whose derivatives tend to zero as $|x|\to \infty$.
\end{thm}

Observe that the condition \lq\lq $a>0$ or \eqref{eq-density1}\rq\rq~ can be summarized as
\begin{equation} \label{eq-density2}
\liminf_{r\downarrow 0} r^{-\beta} \zeta^+ ([-r,r]) > \limsup_{r\downarrow 0} r^{-\beta} \zeta^- ([-r,r]) = 0 \quad \mbox{for some $\beta \in [0,2)$},
\end{equation}
where $\zeta$ denotes the signed measure in the characteristic pair.
Also observe that property \eqref{eq-density1} roughly states that, appropriately scaled,  $\int_{ [-r,r]} x^2 \, \nu^+(\di x)$ dominates $\int_{[-r,r]} x^2 \,\nu^- (\di x)$, which  is in the spirit of the results of  Lemma \ref{lem-de2}.

\begin{proof}
If $a>0$, then the characteristic exponent $\Psi_\mu$ of $\mu$ satisfies $\lim_{|z|\to \infty} z^{-2} \Psi_\mu(z) = -a/2 < 0$ by Lemma \ref{lem-a}. Hence there is $K>0$ such that
$$|\widehat{\mu}(z)| =  \re^{\Re ( \Psi_\mu(z))} \leq \re^{- z^2 a/4} \quad \text{for $\forall \; z \in \bR$ with $|z| \geq K$}.$$
It follows that $\int_\bR |\widehat{\mu}(z)| \, |z|^n \, \di z< \infty$ for all $n\in \bN$, so that $\mu$ has an infinitely often differentiable density on $\bR$ with derivatives tending to $0$ (e.g. \cite[Prop. 28.1]{Sa}).

Now suppose that $a=0$ and that \eqref{eq-density1} holds.
Since $\lim_{r\to 0} r^{-2} (\cos r - 1) = -1/2$ there are $C_1,C_2>0$ and $b>0$ such that
$$C_1 r^{2} \leq 1- \cos r \leq C_2 r^{2}, \quad \forall\; r\in [-b,b].$$
We then conclude for $z\in \bR$ that
\begin{eqnarray*}
\Re (\Psi_\mu(z)) &  = &  \int_\bR (\cos (xz) - 1) \, \nu^+(\di x) + \int_\bR (1- \cos(xz)) \, \nu^-(\di x) \\
& \leq & \int_{|x| \leq b/|z|} (\cos (xz) - 1) \, \nu^+(\di x) \\
& & + \int_{|x|\leq b/|z|} (1-\cos (xz)) \, \nu^-(\di x) + \int_{|x|> b/|z|} (1-\cos (xz)) \, \nu^-(\di x) \\
& \leq & -C_1 z^2 \int_{|x|\leq b/|z|} x^2 \, \nu^+(\di x)  + C_2 z^2 \int_{|x|\leq b/|z|} x^2 \, \nu^-(\di x) + 2 \, \nu^- (\{ x: |x| > b/|z|\}).
\end{eqnarray*}
Denoting the $\liminf$ in \eqref{eq-density1} by $D_1$, we obtain
$$\int_{|x| \leq b/|z|} x^2\, \nu^+(\di x) \geq \frac{D_1}{2} b^\beta |z|^{-\beta} \quad \mbox{and} \quad
\int_{|x|\leq b/|z|} x^2 \, \nu^-(\di x) \leq \frac{C_1 D_1}{4 C_2} b^\beta |z|^{-\beta}$$
for large enough $|z|$, so that
\begin{equation} \label{eq-density3}
\Re (\Psi_\mu(z)) \leq - \frac{C_1 D_1}{4} b^\beta |z|^{2-\beta} + 2 \nu^- (\{x: |x|> b/|z|\}, \quad |z| \; \mbox{large}.
\end{equation} To tackle the last term, write $G(r) := \int_{|x| \leq r} x^2 \, \nu^-(\di x)$ for $r>0$.
 Using partial integration, we can write
\begin{eqnarray*}
\nu^- (\{ x : b/|z| < |x| \leq 1\}) & = & \int_{(b/|z|,1]} x^{-2} \, G(\di x) \\
& = & G(1) - b^{-2} z^2 G(b/|z|) -  \int_{b/|z|}^1 G(x) \, \di x^{-2}.
\end{eqnarray*}
By \eqref{eq-density1}, for every $\varepsilon>0$ we can find $K(\varepsilon) > 0$ such that the above can be bounded from above by
$$G(1) +  \int_{b/|z|}^1 (\varepsilon x^\beta) 2x^{-3}  \, \di x= G(1) - \frac{2\varepsilon}{2-\beta} + \frac{2\varepsilon}{2-\beta} b^{\beta-2} |z|^{2-\beta} , \quad \forall\; |z|\geq K(\varepsilon).$$
Together with \eqref{eq-density3} this implies that there is  $K>0$ such that
$$|\widehat{\mu}(z)| = \exp ( \Re (\Psi_\mu(z))) \leq \exp \left( -\frac{C_1D_1}{8} b^\beta |z|^{2-\beta}\right) , \quad \forall\; |z|\geq K.$$ As in the case $a>0$, this implies that $\int_{\bR} |\widehat{\mu}(z)| \, |z|^n \, \di z < \infty$ for all $n\in \bN,$ giving the claim.
\end{proof}

Turning to continuity, recall that an infinitely divisible distribution is continuous if and only if the Gaussian variance is non-zero or the L\'evy measure is infinite (e.g. \cite[Thm. 27.4]{Sa}). We do not know if an analogous statement holds for quasi-infinitely divisible distributions, but at least we have the following result:

\begin{prop} \label{prop-cont}
Let $\mu$ be a quasi-infinitely divisible distribution with characteristic triplet $(a,\nu,\gamma)_c$ with respect to some $c$.\\
(a) If $a=0$ and $\nu^+(\bR) < \infty$, then $\nu^-(\bR) < \infty$ and $\mu$  is not continuous.\\
(b) Conversely, if $\mu$ is not continuous, then $a=0$, and if additionally $\nu^-(\bR) < \infty$, then $\nu^+(\bR) < \infty$.
\end{prop}

\begin{proof}
Let $X,Y$ and $Z$ be random variables such that $\law(X) = \mu$, such that $\law(Y)$ and $\law(Z)$ are infinitely divisible with
characteristic triplets $(0,\nu^-,0)_c$ and $(a,\nu^+,\gamma)_c$, respectively, and such that $X$ and $Y$ are independent. Then \eqref{eq-factor1} holds.

(a) If $a=0$ and $\nu^+(\bR) < \infty$, then  $Z$ is not continuous by \cite[Thm. 27.4]{Sa}. It follows that neither  $X$ nor $Y$ can be continuous (e.g. \cite[Lemma 27.1]{Sa}), and hence $\nu^-(\bR) < \infty$ (again, \cite[Thm. 27.4]{Sa}).

(b)If $\mu = \law(X)$ is not continuous, then $a=0$ by Theorem \ref{thm-density}.  If additionally $\nu^-(\bR) < \infty$, then $\law(Y)$ is not continuous, hence also  $Z$ is not continuous which implies that  $\nu^+(\bR) < \infty$.
\end{proof}

The fact that $a=0$ together with $\nu^+(\bR) < \infty$ implies $\nu^-(\bR) < \infty$ was already observed in Lemma \ref{lem-de2} (together with the sharper estimate \eqref{eq-measure-estim3}), but here we gave a  different proof of this fact.


\section{Distributions concentrated on the integers} \label{S-integers}

In this section we show in Theorem \ref{thm-integers} that a distribution concentrated on $\bZ$ (i.e.~with support being a subset of $\bZ$) is quasi-infinitely divisible if and only if its characteristic function has no zeroes, thus generalising Theorem \ref{thm-finite}. Unlike the proof of Theorem \ref{thm-finite}, which followed in a somewhat elementary way, the proof of Theorem \ref{thm-integers} is more complicated and uses the Wiener-L\'evy theorem on absolutely summable Fourier series, as well as Theorem \ref{thm-finite}. We shall further characterise weak convergence, moment and support conditions for distributions concentrated on the integers in terms of the characteristic triplet, and obtain sharper results than the general results in Sections \ref{S4} -- \ref{S6}.

Recall that to every continuous function $f\cl \bR \to \bC$ with $f(z) \neq 0$ for all $z\in \bR$ and $f(0) = 1$ there exists a unique continuous function $g$ with $g(0) = 0$ and $\exp (g(z)) = f(z)$ for all $z\in \bR$, called the \emph{distinguished logarithm} of $f$ (e.g. \cite[Lem. 7.6]{Sa}). For a $2\pi$-periodic locally Lebesgue-integrable function $f:\bR \to \bC$, we denote its $n$'th Fourier coefficient by
$$b_n(f) = \frac{1}{2\pi} \int_0^{2\pi} \re^{-\ri n z} f(z) \, \di z, \quad n\in \bZ,$$
and its Fourier series by $\sum_{n\in \bZ} b_n(f) \re^{\ri n z}$. When the Fourier coefficients of $f$ are absolutely summable, then the Fourier series will converge uniformly to $f$, hence $f$ must necessarily be continuous in that case. The set of all $2\pi$-periodic continuous functions $f:\bR \to \bC$ with $\sum_{n\in \bZ} |b_n(f)| < \infty$ forms a commutative Banach algebra with one,  the so-called Wiener algebra ${A} (\mathbb{T})$, where the norm is given by $\|f\|_{{A}(\mathbb{T})} = \sum_{n\in \bZ} |b_n(f)|$, the multiplication is the pointwise multiplication of functions and the one (i.e. the unit) is the function $\one_\bR$ (e.g. Gr\"ochenig \cite[Lem. 5.4]{Groechenig2}).
We now have:

\begin{thm} \label{thm-integers}
Let $\mu = \sum_{n\in \bZ} a_n \delta_n$ be a distribution concentrated on $\bZ$. Then $\mu$ is quasi-infinitely divisible if and only if its characteristic function does not have zeroes. In that case, the Gaussian variance of $\mu$ is zero, the quasi-L\'evy measure $\nu$ of $\mu$ is finite and concentrated on $\bZ$, and the drift lies in $\bZ$. More precisely, if $g\cl \bR \to \bC$ is the distinguished logarithm of $\widehat{\mu}$, then the drift of $\mu$ is $k = (2\pi \ri)^{-1} g(2\pi) \in \bZ$, the function $\widetilde{g}:\bR \to \bC$ defined by $\widetilde{g}(z) = g(z) - \ri k z$ is $2\pi$-periodic, and the quasi-L\'evy measure of $\mu$ is given by $\nu = \sum_{n\in \bZ, n\neq 0} b_n \delta_n$, where
\begin{equation} \label{eq-integers1}
b_n = b_n(\widetilde{g}) = \frac{k}{n} + \frac{1}{2\pi} \int_0^{2\pi} \re^{-\ri n z} g(z) \, \di z \in \bR, \quad n\in \bZ \setminus \{0\},
\end{equation}
is the $n$'th Fourier coefficient of $\widetilde{g}$.
\end{thm}

\begin{proof}
It is clear that the characteristic function of a  quasi-infinitely divisible distribution cannot have zeroes.
Hence we only need to show the converse. Suppose that $\widehat{\mu}$ has no zeroes.
Denote by $g\cl \bR \to \bC$ the distinguished logarithm of $\widehat{\mu}$.
Observe that $\widehat{\mu}(z) = \sum_{n\in \bZ} a_n \re^{\ri nz}$ is $2\pi$-periodic. Hence $\re^{g(2\pi)} = \widehat{\mu}(2\pi) = \widehat{\mu}(0) = 1$ so that $g(2\pi) \in 2\pi \ri \bZ$. Define
$$k= (2\pi \ri)^{-1} g(2\pi)\in \bZ$$
and $\widetilde{g}\cl \bR \to \bC$ by
$\widetilde{g}(z) = g(z) - \ri k z$.
Then $\widetilde{g}$ is continuous, $\widetilde{g}(0) = 0$ and
$$\exp (\widetilde{g}(z)) = \exp (g(z)) \, \exp (-\ri k z) = \widehat{\mu}(z) \, \widehat{\delta_{-k}}(z)
=(\mu \ast \delta_{-k})\widehat{ { } }\; (z).$$
If follows that $\widetilde{g}$ is the distinguished logarithm of the characteristic function
of the discrete distribution $\widetilde{\mu} = \mu \ast \delta_{-k} = \sum_{n\in \bZ} a_n \delta_{n-k}$. Define a $2\pi$-periodic
function $h\cl \bR \to \bC$ by $h(z) = \widetilde{g}(z)$ for $z\in [0,2\pi)$.
Since $\widetilde{g}(2\pi) = 0 = \widetilde{g}(0)$, the function $h$ is continuous.
Since $\widehat{\widetilde{\mu}}$ is $2\pi$-periodic, and $\re^{h(z)} = \re^{\widetilde{g}(z)} = \widehat{\widetilde{\mu}}(z)$ for $z\in [0,2\pi)$ we also have $\re^{h(z)} = \widehat{\widetilde{\mu}}(z)$ for all $z\in \bR$. Hence $h$ is also a distinguished logarithm of $\widehat{\widetilde{\mu}}$, and the uniqueness of the distinguished logarithm gives $h=\widetilde{g}$, consequently $\widetilde{g}$ is $2\pi$-periodic. Since $\widetilde{g}$ is a logarithm of $\widehat{\widetilde{\mu}}$, the fact that $\widetilde{g}(2\pi) = \widetilde{g}(0)$ means that $(\widehat{\widetilde{\mu}}(z))_{z\in [0,2\pi]}$ has index 0 (see \cite[Def. 3.1]{DiBucchianico} for the notion of the index). Denote by $b_n  = b_n(\widetilde{g})$, $n\in \bZ$,  the Fourier coefficients of $\widetilde{g}$, which may be complex.  Since the Fourier coefficients of $\re^{\widetilde{g}}= \widehat{\widetilde{\mu}}$ are absolutely summable (the $m$'th Fourier coefficient is $a_{m+k}$), and since $(\widehat{\widetilde{\mu}}(z))_{z\in [0,2\pi]}$ has index 0, it now follows that also $\sum_{n\in \bZ} |b_n| < \infty$; this is a consequence of the Wiener-L\'evy theorem for holomorphic transformations of functions in the Wiener algebra, and proved in the needed form for the logarithm in Calder\'on et al. \cite[Lemma in Section 2]{Calderon}; see also \cite[Thm. 3.4]{DiBucchianico}. It then follows that
\begin{eqnarray*}
\widehat{\mu}(z) & = & \widehat{\delta_k}(z) \widehat{\widetilde{\mu}}(z) = \re^{\ri k z} \, \re^{\widetilde{g}(z)} \\
& = & \exp \left( \ri k z + \sum_{n\in \bZ} b_n \re^{\ri n z} \right) \\
& = & \exp \left( \ri k z + \sum_{n\in \bZ, n\neq 0} b_n (\re^{\ri n z} - 1)\right) \exp \left(\sum_{n\in \bZ} b_n \right), \quad z\in \bR.
\end{eqnarray*}
Setting $z=0$ in the above equation gives $\exp ( \sum_{n\in \bZ} b_n ) = \widehat{\mu}(0) = 1$, so that $\mu$ is quasi-infinitely divisible with Gaussian variance 0, drift $k$ and quasi-L\'evy measure $\nu = \sum_{n\in \bZ, n\neq 0} b_n \delta_n$, provided we can show that the $b_n$ are real. Since $b_n$ is the $n$'th Fourier coefficient of $\widetilde{g}(z) = g(z) - \ri k z$, it follows that \begin{eqnarray*}
b_n 
& = & \frac{1}{2\pi} \int_0^{2\pi} \re^{- \ri n z } g(z) \, \di z - \frac{\ri k}{2\pi} \int_0^{2\pi} \re^{- \ri n z} z \, dz
 =  \frac{1}{2\pi} \int_0^{2\pi} \re^{- \ri n z } g(z) \, \di z + \frac{k}{n}, \quad n\neq 0,
\end{eqnarray*}
i.e.\ $b_n$ has the form stated in \eqref{eq-integers1}. It remains to show that the Fourier coefficients $b_n$ are real. To do so, observe that the sequence of probability measures \linebreak
$\left((\sum_{n=-m}^m a_n)^{-1} \sum_{m=-n}^m a_n \delta_n\right)_{m\in \bN}$ converges weakly to $\mu$. By modifying the coefficients slightly as in Equation \eqref{eq-modified} in the proof of Theorem \ref{thm-dense}, it follows that there is a sequence $(\mu_m)_{m\in \bN}$ of distributions converging weakly to $\mu$ such that $\mu_m$ is concentrated on $\{-m,\ldots, m\}$ and such that the characteristic function of $\mu_m$ has no zeroes. By Theorem \ref{thm-finite}, each $\mu_m$ is quasi-infinitely divisible with Gaussian variance 0 and quasi-L\'evy measure $\nu_m$ concentrated on $\bZ$. Denote by  $g_m$ the distinguished logarithm of $\widehat{\mu_m}$ and by $k_m$ the drift of $\mu_m$.
Then
$$g_m(z) = \ri k_m z + \sum_{n\in \bZ, n\neq 0} (\re^{\ri n z} - 1) \, \nu_m(\{n\}), \quad z\in \bR,$$
in particular $k_m = (2\pi \ri)^{-1} g_m(2\pi)$ and
\begin{eqnarray*}
b_{m,n} & := & \frac{1}{2\pi} \int_0^{2\pi} \re^{-\ri nz} (g_m(z) - \ri k_m z)\, \di z  \\
& = & \frac{1}{2\pi} \int_0^{2\pi} \re^{-\ri n z} \left( \sum_{j\in \bZ, j\neq 0} (\re^{\ri j z} - 1) \, \nu_m(\{j\}) \right) \, \di z = \nu_m(\{n\}), \quad n\in \bZ \setminus \{0\}.
\end{eqnarray*}
Since $g_m$ converges uniformly on compact subsets of $\bR$ to $g$ as $m\to\infty$ (cf. \cite[Lem. 7.7]{Sa}), also $(z\mapsto g_m(z) - \ri k_m z)_{m\in \bN}$ converges uniformly on compacta to $\widetilde{g}$, hence $\nu_m(\{n\}) = b_{m,n} \to b_n$ as $m\to\infty$ for each $n\in \bZ \setminus \{0\}$. But $\nu_m(\{n\})$ is a real number, hence $b_n$ is real, too. This finishes the proof.
\end{proof}

\begin{cor} \label{c-integers2}
Let $\mu$ be a distribution concentrated on a lattice of the form $r+h\bZ$ with $r\in \bR$ and $h>0$. Then $\mu$ is quasi-infinitely
divisible if and only if its characteristic function has no zeroes. In this case, the quasi-L\'evy measure of $\mu$ is finite and
the Gaussian variance is $0$.
\end{cor}

\begin{proof}
This is exactly as the proof of Corollary \ref{cor-finite3}.
\end{proof}

The following shows that a factor of a quasi-infinitely divisible distribution concentrated on $\bZ$ must necessarily be quasi-infinitely divisible:

\begin{cor} \label{cor-factors}
Let $\mu$, $\mu_1$, $\mu_2$ be distributions on $\R$ such that $\mu=\mu_1*\mu_2$.
Suppose that $\mu$ is quasi-infinitely divisible with $\supp (\mu)\subset \Z$.  Then
$\mu_1$ and $\mu_2$ are quasi-infinitely divisible.
\end{cor}

\begin{proof}
Since $\wh\mu(z)=\wh{\mu_1}(z) \wh{\mu_2}(z)$ and $\wh\mu(z)\ne 0$, neither $\wh{\mu_1}(z)$
nor $\wh{\mu_2}(z)$ have zeroes.  So, it
is enough to show that, for $j=1,2$, there is $b_j\in \R$ such that
$\supp (\mu_j* \dl_{-b_j}) \subset \Z$. Since $\mu$ is discrete, $\mu_1$ and $\mu_2$ are
discrete (\cite[Lem. 27.1]{Sa}).  Choose $b_j\in\R$ such that $\mu_j(\{ b_j\} )>0$.
Let $\mu'_j=\mu_j* \dl_{-b_j}$. Then $\mu'_j(\{ 0\} )>0$ for $j=1,2$ and
$\mu=\mu'_1 *\mu'_2 *\dl_{b_1+b_2}$.  Let $X=X'_1+X'_2 +b_1+b_2$, where $\law (X)=\mu$,
$\law (X'_j)=\mu'_j$ for $j=1,2$ and $X'_1$ and $X'_2$ are independent.  We have
$b_1+b_2\in\Z$, since $P(X=b_1+b_2)\ge P(X'_1=0, X'_2=0)= P(X'_1=0) P( X'_2=0)>0$.
If $\mu_1'(\{ b\})>0$ for some $b\not\in\Z$, then
$P(X=b_1+b_2+b)\ge P(X'_1=b, X'_2=0)= P(X'_1=b) P( X'_2=0)>0$ with $b_1+b_2+b \not\in\Z$ contrary
to the assumption.  Hence $\supp (\mu_1')\subset\Z$.  Similarly $\supp (\mu_2')\subset\Z$.
\end{proof}

It would be interesting to know if factors of arbitrary quasi-infinitely divisible distributions are always quasi-infinitely divisible, which we leave as a topic for further research.

We have seen that although weak convergence of the characteristic pair is sufficient for weak convergence of the quasi-infinitely divisible distribution (Theorem \ref{thm-conv}(a)), it is not necessary (Example \ref{ex-non-conv}), even if the limit distribution is (quasi-) infinitely divisible.
However, for distributions supported on the integers, weak convergence of quasi-infinitely divisible distributions can be characterized
by the weak convergence of the characteristic pair as shown in the following result. Observe that since the quasi-L\'evy measure of a quasi-infinitely divisible distribution supported on the integers is itself supported on $\bZ$ and since the Gaussian variance is 0, the measure $\zeta$ in the characteristic pair coincides with the quasi-L\'evy measure $\nu$ in this case.

\begin{thm} \label{thm-integers-convergence}
 Let $(\mu_m)_{m\in \bN}$ be a sequence of quasi-infinitely divisible distributions concentrated on $\bZ$, $\mu$ a quasi-infinitely divisible distribution concentrated on $\bZ$, $c$ a representation function, and denote the characteristic pairs and triplets of $\mu_m$ and $\mu$ with respect to $c$ by $(\zeta_m,\gamma_m)_c$, $(\zeta,\gamma)_c$, $(0,\nu_m,\gamma_m)_c$ and $(0,\nu,\gamma)_c$, respectively. Denote the drift of $\mu_m$ and $\mu$ by $k_m$ and $k$, respectively. Then the following are equivalent:
\begin{enumerate}
\item[(i)] $\mu_m$ converges weakly to $\mu$ as $m\to\infty$.
\item[(ii)] $k_m$ converges to $k$ as $m\to\infty$ and $\lim_{m\to\infty} \sum_{n\in \bZ} |\nu_m (\{n\}) - \nu(\{n\})| = 0$,
i.e.\ $(\nu_m(\{n\}))_{n\in \bZ}$ converges in ${l}^1$ to $(\nu(\{n\}))_{n\in \bZ}$ as $m\to\infty$.
\item[(iii)] $\gamma_m \to \gamma$ and $\zeta_m \stackrel{w}{\to} \zeta$ as $m\to\infty$.
\end{enumerate}
In particular, for quasi-infinitely divisible distributions $\mu_m$ concentrated on $\bZ$, weak convergence of $\mu_m$ to a quasi-infinitely divisible distribution implies tightness and uniform boundedness of $(\zeta_m)_{m\in \bN}$.
\end{thm}

\begin{proof}
To show that (i) implies (ii), denote the distinguished logarithms of $\widehat{\mu}_m$ and $\widehat{\mu}$ by $g_m$ and $g$, respectively. Then $g_m$ converges uniformly on compact sets to $g$, cf. \cite[Lem. 7.7]{Sa}. Hence $k_m = (2\pi \ri)^{-1} g_m(2\pi) \to (2\pi \ri)^{-1} g(2\pi) = k$ as $m\to\infty$ by Theorem \ref{thm-integers}. Hence also $\mu_m \ast \delta_{-k_m} \stackrel{w}{\to} \mu \ast \delta_{-k}$ as $m\to\infty$, and $\mu_m \ast \delta_{-k_m}$ and $\mu \ast \delta_{-k}$ have drift 0 and quasi-L\'evy measures $\nu_m$ and $\nu$, respectively. Hence, for proving (ii), we will assume that $k_m = k = 0$ for all $m\in \bN$, so that $g_m(2\pi) = g(2\pi) = 0$. Since $\widehat{\mu}_m \to \widehat{\mu}$ uniformly (both are $2\pi$-periodic), we have $\sup_{z\in\bR} \left| (\widehat{\mu}_m (z) / \widehat{\mu}(z)) -1 \right| < 1/2$ for
large enough $m$. Then for large $m$, the logarithmic expansion
\begin{equation} \label{eq-log-Wiener}
h_m(z) := -\sum_{n=1}^\infty \frac{1}{n} \left( 1-\frac{\widehat{\mu}_m (z)}{\widehat{\mu}(z)} \right)^n , \quad z\in \bR,
\end{equation}
of the principal branch of the logarithm of $\widehat{\mu}_m (z) / \widehat{\mu}(z)$ converges uniformly.
Then
$$\exp (h_m(z)) = \frac{\widehat{\mu}_m(z)}{\widehat{\mu}(z)} = \exp (g_m(z) - g(z)), \quad z\in \bR,$$
for large $m$, and since $h_m$ is continuous with $h_m(0) = 0$, as is $g_m - g$, the uniqueness of the distinguished logarithm shows that
\begin{equation} \label{eq-log-Wiener2}
h_m(z) = g_m (z) - g(z), \quad \forall\; z\in \bR\q \text{ for $m$ large.}
\end{equation}
Write $\mu = \sum_{n\in \bZ} a_n \delta_n$ and $\mu_m = \sum_{n\in \bZ} a_{n,m} \delta_n$. Since $\mu_m \stackrel{w}{\to} \mu$ as $m\to\infty$ we have $a_{n,m} \to a_n$ for each $n\in \bZ$ as $m\to\infty$, and since $\sum_{n\in \bZ} a_{n,m} = \sum_{n\in \bZ} a_n = 1$ and all coefficients are non-negative, it follows that also $\sum_{n\in \bZ} |a_{n,m} - a_n| \to 0 $ as $m\to\infty$.
But $\widehat{\mu}_m(z) = \sum_{n\in \bZ} a_{n,m} \re^{\ri n z}$ and $\widehat{\mu}(z) = \sum_{n\in \bZ} a_n \re^{\ri n z}$, hence $a_{n,m} = b_n (\widehat{\mu}_m)$ and $a_n = b_n(\widehat{\mu})$. Altogether, we conclude that $\widehat{\mu}_m$ converges to $\widehat{\mu}$ in the $A(\mathbb{T})$-norm. Since $A(\mathbb{T})$ is a Banach algebra, this also implies that $\widehat{\mu}_m / \widehat{\mu}$ converges to 1 in the $A(\mathbb{T})$-norm as $m\to\infty$. In particular, for each $\varepsilon \in (0,1)$, there is $N(\varepsilon) \in \bN$ such that
$\| 1- (\widehat{\mu}_m / \widehat{\mu}) \|_{A(\mathbb{T})} < \varepsilon$ for all $m\geq N(\varepsilon)$, so that the series defining $h_m$ in \eqref{eq-log-Wiener} converges also in the $A(\mathbb{T})$-norm (to the same limit, since $\sup_{z\in \bR} |\psi(z)|\leq \|\psi\|_{A(\mathbb{T})} $ for $\psi \in A(\mathbb{T})$) and we have $\|h_m\|_{A(\mathbb{T})} \leq \sum_{n=1}^\infty n^{-1} \varepsilon^n \leq \varepsilon / (1-\varepsilon)$ for $m\geq N(\varepsilon)$. Using \eqref{eq-log-Wiener2} this means that $g_m -g$ converges to 0 and hence $g_m$ to $g$ in the $A(\mathbb{T})$-norm as $m\to\infty$. By Theorem \ref{thm-integers} this means that $(\nu_m(\{n\})_{n\in \bZ}$ converges in $l^1$ to $(\nu(\{n\}))_{n\in \bZ}$ as $m\to\infty$, which finishes the proof of (ii).

To see that (ii) implies (iii), observe that $\zeta_m = \nu_m$ and $\zeta= \nu$ since the quasi-L\'evy measures are concentrated on $\bZ$. The $l^1$-convergence of the quasi-L\'evy measures then obviously implies $\zeta_m \stackrel{w}{\to} \zeta$ as $m\to\infty$ and $\gamma_m = k_m + \sum_{n\in \bZ} c(n) \nu_m(\{n\}) \to k + \sum_{n\in \bN} c(n) \nu(\{ n\})$ as $m\to\infty$, which is (iii). That (iii) implies (i) follows from Theorem \ref{thm-conv}(a); observe that we do not need  $c$ to be continuous, since we can always modify $c$ between two integers in order to make it continuous without affecting the integrals, since the quasi-L\'evy measures are supported only on $\bZ$.

Finally, tightness and uniform boundedness of $(\zeta_m)_{m\in \bZ}$ follows from (iii).
\end{proof}

We have seen that the quasi-L\'evy measure of a quasi-infinitely divisible distribution on $\bZ$ is finite, the drift an integer and the Gaussian variance 0. There is also a converse:

\begin{thm} \label{thm-integers3}
Let $\mu$ be a quasi-infinitely divisible distribution on $\bR$. Then the following are equivalent:
\begin{enumerate}
\item[(i)] $\mu$ is concentrated on the integers, i.e. $\supp \mu \subset \bZ$.
\item[(ii)] The quasi-L\'evy measure of $\mu$ is concentrated on $\bZ$, the drift is an integer and the Gaussian variance is $0$.
\end{enumerate}
\end{thm}

\begin{proof}
That (i) implies (ii) is Theorem \ref{thm-integers}. For the converse, denote the drift of $\mu$ by $\gamma$ and its quasi-L\'evy measure by $\nu$. Let $X,Y,Z$ be random variables such that $\law(X) = \mu$, $\law(Y)$ is infinitely divisible with characteristic triplet $(0,\nu^-,0)_0$, $\law(Z)$ is infinitely divisible with characteristic triplet $(0,\nu^+,\gamma)_0$, and such that $X$ and $Y$ are independent. Then \eqref{eq-factor1} is satisfied. By \cite[Cor.\ 24.6]{Sa}, $Y$ and $Z$ are concentrated on $\bZ$. Hence also $X$ must be concentrated on $\bZ$, i.e. $\supp \mu \subset \bZ$ and we are done.
\end{proof}

Denote by $D= \{ w\in \bC\cl |w| < 1\}$ the open unit disk and by $\overline{D} = \{ w \in \bC \cl |w| \leq 1\}$ the closed unit disk. A special case of quasi-infinitely divisible distributions is formed by the \emph{discrete pseudo-compound Poisson distributions}, in short \emph{DPCP}-distribution, which have applications in insurance mathematics. Following Zhang et al. \cite[Def. 5.1]{ZhangLiuLi}, a DPCP-distribution is a distribution $\mu = \sum_{n=0}^\infty a_n \delta_n$ on the non-negative integers whose probability generating function $\overline{D} \ni w \mapsto \sum_{n=0}^\infty a_n w^n$ has the form
\begin{equation} \label{eq-DPCP}\sum_{n=0}^\infty a_n w^n = \exp \left( \sum_{j=1}^\infty \alpha_j \lambda (w^j -1) \right) ,  \quad \forall\; w\in \overline{D},
\end{equation}
for some $\lambda>0$ and a sequence $(\alpha_j)_{j\in \bN}$ of real numbers such that $\sum_{j=1}^\infty |\alpha_j| < \infty$ and $\sum_{j=1}^\infty \alpha_j = 1$. Setting $w= \re^{\ri z}$, $z\in \bR$, it is clear that a DPCP-distribution is quasi-infinitely divisible with drift 0, Gaussian variance 0 and quasi-L\'evy measure $\lambda \sum_{j=1}^\infty \alpha_j \delta_j$. Zhang et al. \cite{ZhangLiuLi} obtained the following characterisation of DPCP-distributions:

\begin{thm} [Zhang et al.\ \cite{ZhangLiuLi}, Thm.\ 5.2] \label{thm-Zhang}
A distribution $\mu = \sum_{n=0}^\infty a_n \delta_n$ is a DPCP-distribution if and only if the probability generating function has no zeroes on $\overline{D}$, i.e.\ if\/ $\sum_{n=0}^\infty a_n w^n \neq 0$ for all $w\in \overline{D}$.
\end{thm}

It follows from Theorem \ref{thm-Zhang} that a DPCP-distribution must necessarily have an atom at 0.
The following theorem establishes the precise connection to quasi-infinitely divisible distributions.

\begin{thm} \label{thm-DPCP-quasi}
Let $\mu=\sum_{n\in \bZ} a_n \delta_n$ be a distribution on $\bZ$ and let $k\in \bZ$. Then the following are equivalent:
\begin{enumerate}
\item[(i)] $\mu$ is quasi-infinitely divisible with drift $k$, quasi-L\'evy measure $\nu$ and $\supp \nu \subset \bN$.
\item[(ii)] $\mu$ is quasi-infinitely divisible with drift $k$, quasi-L\'evy measure $\nu$ and $\supp \nu^+ \subset \bN$.
\item[(iii)] $a_k\neq 0$, $a_n =0$ for $n< k$ (i.e. $\inf (\supp \mu) = k$) and the function $\overline{D}\to \bC$
given by $w\mapsto \sum_{n=0}^\infty a_{n+k} w^n$ has no zeroes on $\overline{D}$.
\item[(iv)] $\mu \ast \delta_{-k}$ is a DPCP-distribution, in particular is concentrated on $\bN_0$.
\item[(v)] $a_k \neq 0$, $a_n=0$ for $n<k$, and there exists a sequence $(q_n)_{n\in \bN}$ of real numbers with $\sum_{n=1}^\infty |q_n| < \infty$ and such that
    \begin{equation} \label{eq-Katti}
    n a_{n+k} = \sum_{j=1}^n j q_j a_{n+k-j}, \quad \forall\; n\in\bN.
    \end{equation}
\end{enumerate}
Further, the sequence $(q_n)_{n \in \bN}$ appearing in (v) is related to the quasi-L\'evy measure $\nu$ of $\mu$
by $q_n = \nu (\{n\})$ for all $n\in \bN$.
\end{thm}

\begin{proof}
The equivalence of (iii) and (iv) is Theorem \ref{thm-Zhang}, and that (iv) implies (i) has been observed  after the definition of DPCP-distributions.
That (i) implies (ii) is trivial, and that (ii) implies (i) follows from Proposition \ref{prop-support1} and Theorem \ref{thm-integers}. Let us prove that (i) implies (iii). Again, by Proposition \ref{prop-support1} (and since the Gaussian variance is 0), $k = \inf (\supp \mu)$. Define the functions $f,g: \overline{D} \to \bC$ by
$$f(w) = \sum_{n=0}^\infty a_{n+k} w^n \quad \mbox{and} \quad g(w) = \exp \left( \sum_{n=1}^\infty (w^n-1) \, \nu(\{n\}) \right) .$$ Then both $f$ and $g$ are holomorphic on $D$ and continuous on $\overline{D}$, in particular bounded on $\overline{D}$. Since
$$g(\re^{\ri z}) = (\mu \ast \delta_{-k}) \widehat{ { } } \; (z) = ( \sum_{n=0}^\infty a_{n+k} \delta_n ) \widehat{ { } }\; (z) = f(\re^{\ri z}), \quad z\in \bR,$$
$f$ and $g$ agree on the boundary $\partial D = \{ w \in \bC: |w|=1\}$ and hence $f=g$ on $\overline{D}$, see e.g. \cite[Thm. 13.5.3]{Conway2}. Since $g$ has no zeroes on $\overline{D}$, the same is true for $f$.
We have proved the equivalence of conditions (i) -- (iv). For proving that (i) - (iv) are equivalent to (v), by considering $\mu \ast \delta_{-k}$ we can and shall assume without loss of generality that $k=0$ so that $a_0 \neq 0$ and $a_n = 0$ for $n < 0$. The equivalence of (i) and (v) and the relation $q_n = \nu(\{n\})$ then follows in complete analogy to the proof of Corollary 51.2 in \cite{Sa}, with the help of Theorem \ref{thm-support3}.
\end{proof}

Condition (v) in Theorem \ref{thm-DPCP-quasi} is a version of Katti's criterion for quasi-infinitely divisible distributions, and appears also under the name of Panjer-recursions. The equivalence of (iv) and (v) above (without explicitly stated summability conditions on $(q_n)$) has already been observed  by H\"urlimann \cite[Lem.\ 1]{Huerlimann1989}. Observe that \eqref{eq-Katti} gives an easy method of determining the quasi-L\'evy measure of a distribution that satisfies the equivalent conditions of Theorem \ref{thm-DPCP-quasi}, by simply solving \eqref{eq-Katti} recursively for $q_n$.

%

In Example \ref{ex-non-moments} we have seen that existence of certain moments cannot always be characterised by the corresponding property of the quasi-L\'evy measure. Now we show that for quasi-infinitely divisible distributions on the integers and for submultiplicative functions satisfying an additional condition, this is possible. We need
the following Wiener-L\'evy type theorem for the Beurling-algebra of $2\pi$-periodic functions whose Fourier-coefficients are summable with respect to a given weight satisfying the GRS-condition. It can be (almost) found in this form in Bhatt and Dedania \cite{Bhatt}:

\begin{thm} [Bhatt and Dedania \cite{Bhatt}] \label{thm-Bhatt}
Let $h\cl \bZ \to [0,\infty)$ be a submultiplicative function, i.e. such that there exists $B>0$ with
\begin{equation} \label{eq-subm2}
h(n+m) \leq B h(n) h(m), \quad \forall\; n,m\in \bZ.
\end{equation}
Assume furthermore that
$h$ satisfies the Gelfand-Raikov-Shilov (GRS)-condition
\begin{equation} \label{eq-GRS}
\lim_{n\to\pm \infty} \frac{\log h(n)}{n} = 0.
\end{equation}
Let $f$ be a continuous $2\pi$-periodic complex valued function such that its Fourier coefficients $b_n(f)$ satisfy $\sum_{n\in \bZ} h(n) |b_n(f)| < \infty$, and let $F:U\to \bC$ be  a holomorphic function defined in an open neighbourhood $U$ of the range of $f$. Then the Fourier coefficients $b_n(F\circ f)$ of  $F\circ f$ satisfy $\sum_{n\in \bZ} h(n) |b_n (F \circ f)|$, too.
\end{thm}

\begin{proof}
Multiplying \eqref{eq-subm2} by $B$ we have $Bh(m+n) \leq (Bh(n)) (Bh(m))$. By replacing $h$ by $Bh$, we may hence assume that $B=1$. Then, by submultiplicativity $\log h(nm)\leq n \log [h(m)]$ for $n\in \bN$ and $m\in \bZ$, so that the GRS-condition implies $\log h(m) \geq 0$ for each $m\in \bZ$, i.e. $h(m) \geq 1$. With these additional hypothesis, the theorem is then stated in Bhatt and Dedania \cite{Bhatt}, observing that the function $\chi$ there can be chosen to be the original weight-function $h$ (in their notation, $\omega$) as pointed out in their proof, since $\inf \{ [h(n)]^{1/n} : n\in \bN\} = \sup \{ [h(n)]^{1/n} : -n \in \bN\} = 1$ by the GRS-condition. Since the proof in \cite{Bhatt} is a bit short for people that are not familiar with the Gel`fand theory, an alternative reasoning can be based on Gr\"ochenig \cite{Groechenig2}: by Corollary 5.27 in \cite{Groechenig2}, the Beurling algebra under consideration is inverse closed in the algebra of continuous  $2\pi$-periodic functions. Since it is further continuously embedded into that algebra as is easy to see, the Riesz-calculi for holomorphic functions in both algebras coincide \cite[Cor. 5.15]{Groechenig2} and hence exactly the same proof as in \cite[Thm. 5.16]{Groechenig2} gives the claim.
\end{proof}

Theorem \ref{thm-Bhatt} can be applied to the distinguished logarithm and we obtain the following analogue to the Lemma on page 491 of Calder\'on et al. \cite{Calderon}:

\begin{cor} \label{cor-Bhatt}
Let $h\cl \bZ \to [0,\infty)$ be a submultiplicative function satisfying the GRS-condition and $f\cl \bR \to \bC$ a continuous
$2\pi$-periodic complex-valued function such that $f(z) \neq 0$ for all $z\in \bR$ and such that the Fourier coefficients $b_n(f)$ satisfy $\sum_{n\in \bZ} h(n) |b_n(f)| < \infty$. Assume furthermore that the distinguished logarithm $g$ of $f$ satisfies $g(2\pi) = g(0)$. Then the Fourier coefficients $b_n(g)$ of $g$ satisfy $\sum_{n\in \bZ} h(n) |b_n(g)| < \infty$, too.
\end{cor}

\begin{proof}
As in the proof of Theorem \ref{thm-Bhatt}, we can and do assume that $B$ in \eqref{eq-subm2} is equal to 1. Then the space $A_h=A_h(\T)$ of all $2\pi$-periodic complex-valued continuous functions $\varphi$ on $\bR$ with $\sum_{n\in \bZ} h(n) |b_n(\varphi)| < \infty$ is a Banach algebra under the usual addition and multiplication of functions, and with norm given by $\|\varphi\|_h = \sum_{n\in \bZ} h(n) |b_n(\varphi)|$, cf. \cite[Lem. 5.22]{Groechenig2}. In particular, $f \in A_h$ by assumption.

From the proof of the Lemma in \cite[p. 491]{Calderon} it
follows that there is a trigonometric polynomial $p(z)$, say $p(z) = \sum_{n=-m}^m q_n \re^{\ri n z}$, such that the range of $z \mapsto \varphi_1(z) := \exp (- \ri p(z)) f(z)$ lies in the half-plane $\{ w\in \bC\cl \Re (w) > 0\}$. Since obviously
$p \in A_h$, Theorem \ref{thm-Bhatt} gives $\re^{- \ri p(\cdot)} \in A_h$, hence by the Banach-algebra property also $\varphi_1 \in A_h$. Denote by $\log$ the principal branch of the logarithm and define $\psi_1(z) = \log \varphi_1(z)$ for $z\in \bR$.
Then as in \cite[p. 491]{Calderon}, but using Theorem \ref{thm-Bhatt} instead of the Wiener-L\'evy theorem applied to the principal branch $\log$ of the logarithm, it follows that $\psi_1\in A_h$. Since $\psi_1$ and $p$ are continuous $2\pi$-periodic functions with
$$\exp (\psi_1(z) + \ri p(z)) = \varphi_1(z) \exp ( \ri p(z)) = f(z) = \exp (g(z)), \quad \forall\; z \in \bR,$$
the uniqueness of the distinguished logarithm shows that there is $l\in \bZ$ such that $g(z) = \psi_1(z) + \ri p(z) + 2\pi \ri l$. Since $\psi_1$, $\ri p(\cdot)$ and constant functions are in $A_h$, it follows that also $g\in A_h$, which is the claim.
\end{proof}

With Corollary \ref{cor-Bhatt} we can now characterise finiteness of $h$-moments of quasi-infinitely divisible distributions on the integers in terms of the quasi-L\'evy measure, provided $h$ satisfies the GRS-condition:

\begin{thm} \label{thm-moments-integers}
Let $\mu$ be a quasi-infinitely divisible distribution on $\bZ$ with quasi-L\'evy measure $\nu$, and let $h\cl \bZ \to [0,\infty)$ be a submultiplicative weight-function that satisfies the GRS-condition, i.e. $h$ satisfies \eqref{eq-subm2} and \eqref{eq-GRS}. Then the following are equivalent:
\begin{enumerate}
\item[(i)] $\mu$ has finite $h$-moment, i.e. $\int_\bR h(x) \, \mu(\di x) < \infty$.
\item[(ii)] $\nu^+$ has finite $h$-moment, i.e. $\int_\bR h(x) \, \nu^+(\di x) < \infty$.
\item[(iii)] $|\nu|$ has finite $h$-moment, i.e. $\int_\bR h(x) \, |\nu|(\di x) < \infty$.
\end{enumerate}
\end{thm}

\begin{proof} That (iii) implies (ii) is clear, and that (ii) implies (i) follows from Theorem \ref{thm-moments1}, by observing that every submultiplicative function $h$ on $\bZ$ can be extended to a submultiplicative function on $\bR$ by setting $h(x) := \max \{ h(\lfloor x \rfloor), h(\lceil x \rceil)\}$ for $x\in \bR$, where $\lfloor x \rfloor$ denotes the largest integer smaller than or equal to $x$, and $\lceil x \rceil$ the smallest integer greater than or equal to $x$. It remains to show that (i) implies (iii). For that, let $\mu$ be with drift $k$.
Since
$$h(n-k) \leq B h(n) h(-k) \quad \mbox{and} \quad h(n) \leq B h(n-k) h(k), \quad \forall\; n\in \bZ,$$
it follows that $\mu$ has finite $h$-moment if and only if $\mu \ast \delta_{-k}$ has finite $h$-moment. Since further $\mu$ and $\mu\ast \delta_{-k}$ have the same quasi-L\'evy measure, we can and do assume without loss of generality that $k=0$. Denote by $g$ the distinguished logarithm of $\widehat{\mu}$. From Theorem \ref{thm-integers} we know that $g(2\pi) = g(0) = 0$ and that $\nu(\{n\})$ is the $n$'th Fourier coefficient of $g$. The claim then follows directly from Corollary \ref{cor-Bhatt}, since $\mu(\{n\})$ is the $n$'th Fourier coefficient of $\widehat{\mu}$ and since $\mu$ has finite $h$-moment.
\end{proof}

Theorem \ref{thm-moments-integers} applies in particular to the submultiplicative functions $x\mapsto (|x| \vee 1)^\alpha$ for $\alpha > 0$, $x\mapsto \log (|x|\vee \re)$ and $x\mapsto \exp (\alpha |x|^\beta)$ for $\alpha > 0$ and $\beta \in (0,1)$ since they satisfy the GRS-condition, but not to $x\mapsto \re^{\alpha |x|}$ or $x\mapsto \re^{\alpha x}$ for $\alpha > 0$ since they do not satisfy the GRS-condition.

\bigskip

\noindent
Alexander Lindner, Lei Pan\\
Ulm University,
Institute of Mathematical Finance,
Helmholtzstra{\ss}e 18,
89081 Ulm,
Germany\\
emails: alexander.lindner@uni-ulm.de, lei.pan@uni-ulm.de
\smallskip

\noindent
Ken-iti Sato\\
Hachiman-yama 1101-5-103, Tenpaku-ku, Nagoya, 468-0074 Japan\\
email: ken-iti.sato@nifty.ne.jp

\end{document}